\documentclass[12pt,reqno,a4paper]{amsart}

\usepackage{fullpage, color}
\usepackage[T1]{fontenc}
\usepackage{amsmath,amssymb,mathtools,mathrsfs, enumerate}

\usepackage{soul}

\usepackage[colorlinks=true,urlcolor=cyan,linkcolor=blue,citecolor=magenta]{hyperref}

\DeclareTextFontCommand{\emph}{\bfseries}

\newtheorem{theorem}{Theorem}[section]
\newtheorem{proposition}[theorem]{Proposition}
\newtheorem{lemma}[theorem]{Lemma}
\newtheorem{corollary}[theorem]{Corollary}
\newtheorem{conj}[theorem]{Conjecture}

\theoremstyle{definition}
\newtheorem{definition}[theorem]{Definition}

\theoremstyle{remark}
\newtheorem{remark}[theorem]{Remark}

\newcommand{\C}{\mathbb C}
\newcommand{\R}{\mathbb R}
\newcommand{\Q}{\mathbb Q}
\newcommand{\Z}{\mathbb Z}
\newcommand{\N}{\mathbb N}
\newcommand{\PP}{\mathbb P}
\newcommand{\Sym}{\operatorname{Sym}}
\newcommand{\GL}{\operatorname{GL}}

\newcommand{\vol}{\operatorname{vol}}
\newcommand{\diam}{\operatorname{diam}}

\newcommand{\Tet}{\operatorname{Tet}}
\newcommand{\spn}{\operatorname{span}}

\newcommand{\tr}{\mathsf T}
\newcommand{\hh}{\mathrm h}
\newcommand{\ee}{\mathrm e}

\title{Degenerating orbits of the Longest Edge Bisection process}

\author[Karim Adiprasito]{Karim A.~Adiprasito}
\address{Karim Adiprasito, Sorbonne Université and Université Paris Cité, CNRS, IMJ-PRG, F-75005 Paris}
\email{adiprasito@imj-prg.fr}

\author{Daniel Kalmanovich}
\address{Daniel Kalmanovich, Department of Computer Science, Ariel University} 
\email{danielk@ariel.ac.il}
\urladdr{https://danielkalmanovich.wixsite.com/website}

\author{Yaar Solomon}
\address{Yaar Solomon, Department of Mathematics, Ben-Gurion University of the Negev} 
\email{yaars@bgu.ac.il}
\urladdr{https://www.math.bgu.ac.il/~yaars/}

\date{}

\begin{document}

\begin{abstract}
We study the Longest Edge Bisection (LEB) process as a dynamical system on the projective shape space of simplices. A long-standing conjecture going back to Adler and Rivara-Levin and motivated by finite-element mesh refinement, often taken as a standing assumption, is that this procedure is non-degenerate and, in fact, in a certain way periodic.
We prove:
\begin{itemize}
\item There are 3-dimensional simplices such that the longest edge-bisection algorithm degenerates.
\item There is an open set of 4-dimensional simplices on which the longest edge-bisection algorithm degenerates.
\item If parametrizing the space of $d$-dimensional simplices by independent standard Gaussian vectors, then as $d$ increases, a random simplex degenerates asymptotically almost surely.
\end{itemize}
This is realized through exhibiting hyperbolic behaviour of the LEB process. We also exhibit elliptic behaviour that is nonperiodic.
\end{abstract}

\maketitle

\section{Introduction}\label{sec:introduction}

The Longest Edge Bisection (LEB) algorithm is one of the most basic
refinement procedures in mesh generation, of central importance in particular for implementations of the finite-element method.
In dimension two, one
bisects a triangle by joining the midpoint of one of its longest edges
to the opposite vertex, and then repeats this operation on the
descendant triangles iteratively. Rosenberg and Stenger proved that the
angles of all descendants are bounded from below in terms of the
smallest angle of the initial triangle~\cite{RosenbergStenger1975}.
Kearfott proved convergence of the mesh diameter in arbitrary
dimension~\cite{Kearfott1978}, and Stynes obtained sharp planar
convergence results~\cite{Stynes1979,Stynes1980}. A particularly
beautiful planar phenomenon is the finite similarity classes (FSC)
theorem: every initial triangle generates only finitely many
similarity classes. This was proved by Stynes~\cite{Stynes1980},
reproved by Adler~\cite{Adler83}, and later placed in a natural
hyperbolic dynamical framework by Perdomo and
Plaza~\cite{PerdomoPlaza2014}.

In dimension $3$, the analogous procedure bisects a tetrahedron
through the midpoint of a longest edge. Each selected edge gives two
children, and ties between longest edges introduce an additional
choice. The basic qualitative questions are whether repeated
refinement can produce tetrahedra of arbitrarily poor shape and
whether, as in dimension two, only finitely many similarity classes
can occur. Adler briefly pointed to a family of nearly equilateral
tetrahedra that should generate at most $37$ similarity classes, but
his tetrahedral observation was stated without a complete proof
\cite{Adler83}. This construction was made explicit computationally
by Su\'arez, Trujillo, and Moreno~\cite{SuarezTrujilloMoreno2021}, and
the nearly equilateral family was subsequently proved to have at most
$37$ similarity classes by Trujillo-Pino, Su\'arez, and
Padr\'on~\cite{TrujilloPino2024}. Rivara and Levin explicitly investigated and asked whether the 3-dimensional case of the algorithm degenerates \cite{RL1992}.

Michaud and Korotov~\cite{MichaudKorotov2026}
develop a normalized tetrahedral shape space and prove finiteness
for several special orbits and perturbative families. Michaud and
Korotov also report extensive numerical evidence suggesting
nondegeneracy for the examples they study, see also the references
therein, including
\cite{HannukainenKorotovKrizek2014, KorotovKrizekKropac2008,
PadronPlazaSuarez2023, PadronTrujilloSuarez2025}.

Our paper shows that the global $3$-dimensional picture is
different.

Our point of view is the one from dynamical systems. The moduli space of labeled
tetrahedra up to similarity is identified with the projectivization
of the cone of positive-definite Gram matrices, or equivalently with
the projectivized tetrahedral cone in squared-edge coordinates.
Directed bisections, together with relabelings at the ends of
bisection blocks, act on this space by projective linear
transformations. The requirement that the selected edge be longest
restricts each transformation to an explicit admissibility domain,
so that LEB becomes a partially defined semigroup action. The constructions in this paper arise from two qualitatively different invariant dynamics for this action: the first is a
hyperbolic return map on a one-parameter family, contracting toward
the degenerate boundary. One could now think that orbits that do not degenerate are finite; this is not the case. The second construction is an elliptic return map on a compact
invariant ellipse, conjugate to an irrational rotation. 

In the LEB process, at every stage any longest edge may be selected, and either child may be followed. 
If several longest edges are tied, one needs to choose which one to bisect. The \emph{opposite-edge tie-breaking rule} resolves ties by selecting a longest edge whose opposite edge has the maximum length\footnote{This is not a deterministic rule, but in all of our constructions, it is.}. The comparison of opposite
edges is closely related to the condition emphasized in the Adler
family, where the two largest edges are opposite
\cite{Adler83,SuarezTrujilloMoreno2021,TrujilloPino2024}. Here,
however, it is used as a tie-breaking rule, and the required strict
opposite-edge inequality is verified directly along the branch. This is only auxiliary: the final construction does not need tie-breaking.

To formulate degeneration in a scale-invariant way, for a nondegenerate $d$-simplex $\Delta$, put
\begin{equation}\label{eq:degeneration_parameter}
    \vartheta_d(\Delta)
    :=\frac{\vol_d(\Delta)}{\diam(\Delta)^d}.
\end{equation}
For tetrahedra, we abbreviate $\vartheta_3$ to $\vartheta$. The function $\vartheta_d$ is invariant under similarities. A branch
$(\Delta_n)_{n\geq0}$ \emph{degenerates} if
\[    \vartheta_d(\Delta_n)\xrightarrow{n\to\infty}0,
\]
that is, if it approaches the degenerate boundary after normalization
by diameter; it captures the intuition of becoming flat, or more formally, converging to the boundary of the configuration space of simplices.

Our first result is the illustrative starting point of the paper.

\begin{theorem}
\label{thm:main-degeneration}
There exist uncountably many pairwise non-similar tetrahedra, each having a degenerating (infinite) LEB-branch, with ties resolved by the opposite-edge tie-breaking rule.
\end{theorem}

A similar result was also proven independently by Korotov \cite{Korotov}, who also relied on tie-breaking. One may still ask, however, whether this phenomenon relies on tie-breaking. We show that this is not the case.

\begin{theorem}
\label{thm:main-degeneration2}
There exist uncountably many pairwise non-similar tetrahedra, each having a degenerating (infinite) LEB-branch independent of a tie-breaking rule; each longest edge in the degenerating branch is unique.
\end{theorem}

One may further ask, then, whether this is a very special property, perhaps confined to special algebraic sets of simplices. The answer is, again, no.

\begin{theorem}
\label{thm:main-degeneration3}
There is an open set of non-similar $4$-simplices, each having a degenerating (infinite) LEB-branch independent of a tie-breaking rule; each longest edge in the degenerating branch is unique.
\end{theorem}

A particular corollary is then as follows:

\begin{corollary}
\label{cor:higher-dimensional-fsc}
Let \(\Sigma_d\) be the simplex whose \(d+1\) vertices are independent standard Gaussian vectors in \(\mathbb R^d\). There is an absolute constant \(c>0\) such that, for every \(d\geq 4\),
\[ \Pr\!\left(\Sigma_d\text{ has a degenerating LEB branch}\right) \geq 1-e^{-cd}. \]

In fact, this lower bound holds for the probability that \(\Sigma_d\) belongs to the interior of the set of simplices having a degenerating branch. Moreover, its entire refinement tree is tie-free almost surely.
\end{corollary}

Let us make the following conjecture, motivated by nothing more than vague intuition:
\begin{conj}
Almost every simplex of dimension $d\ge d_0$ has a degenerating branch under longest edge bisection. 
\end{conj}

Indeed, once we identified the general idea, the examples for Theorem~\ref{thm:main-degeneration3} were found by writing a script and letting it run overnight; degenerate orbits seem to be rather abundant.

Back to dimension $3$, the failure of finite similarity classes is not confined to branches
which approach the degenerate boundary. We provide another
construction with a nondegenerate infinite branch. This construction
is elliptic rather than hyperbolic.

\begin{theorem}
\label{thm:main-elliptic}
There exists a nonempty open subset of the tetrahedral shape space such that every class in this subset admits an LEB-branch which remains in a compact subset of the nondegenerate shape space and contains infinitely many similarity classes.
Moreover, this open set contains a dense full-measure subset whose elements' entire tree is tie-free.

\end{theorem}

Similar analogues of the above results (Theorems~\ref{thm:main-degeneration3} and Corollary~\ref{cor:higher-dimensional-fsc}) for non-degenerating sets exist:
\begin{itemize}
\item there is an open set of 4-simplices with an infinite LEB branch and
\item as dimension increases, most simplices have an infinite branch in a compact subset.
\end{itemize}

The paper is organized as follows.
Section~\ref{sec:setup} introduces a dynamical setup that is used throughout the paper to study the space of tetrahedral shapes, including the Gram-matrix and squared-edge
coordinates models. We also introduce a convenient way to navigate the branching tree inherent to LEB by using (marked) words.

Section~\ref{sec:hyperbolic} explains the hyperbolic mechanism and
proves Theorem~\ref{thm:main-degeneration} using a marked word of length two. We then prove Theorem~\ref{thm:main-degeneration2} using a marked word of length 32. A small modification then yields Theorem~\ref{thm:main-degeneration3}.

Section~\ref{sec:elliptic} explains the elliptic mechanism and proves
Theorem~\ref{thm:main-elliptic} using a marked word of length three
acting as an irrational rotation on a compact invariant ellipse.

\subsection*{Acknowledgements}

We used an LLM for computational exploration, algebraic checks, and editorial assistance. The mathematical ideas, research direction, and the dynamical approach were developed by us. K. A. is supported by Horizon Europe ERC Grant number: 101045750 / Project acronym: HodgeGeoComb.

\section{Notation and dynamical setup}\label{sec:setup}
The goal of this section is to develop a dynamical approach to studying the LEB process. To simplify notation, the setup is given for dimension $3$, though all components may be defined for general dimension $d$ in a similar way. We introduce two isomorphic models for the tetrahedral shape space: the Gram-matrix model and the squared-edge length model. We then describe the bisections as linear transformations and, accounting for LEB-admissibility, we model the LEB process as a linear partial action on this space.   

The constructions in Sections~\ref{sec:hyperbolic} and
\ref{sec:elliptic} use both models, but in different ways. In the hyperbolic construction in Sections~\ref{sec:hyperbolic}, the
invariant family is obtained in Gram-matrix coordinates, while
squared-edge coordinates are used to verify the longest-edge
conditions and the tie-breaking rule. In the elliptic construction in Section \ref{sec:elliptic},
the invariant family and its dynamics are described directly in
squared-edge coordinates, while the Gram matrix is used to verify
nondegeneracy.

Let $M_3(\R)$ denote the vector space of $3\times3$ real matrices, and
write
\[
    \Sym_3:=\{G\in M_3(\R)\mid G^{\tr}=G\},
    \qquad
    \Sym_3^{>0}:=\{G\in\Sym_3\mid G>0\},
\]
where $G>0$ means that $G$ is positive definite, i.e., $x^{\tr}Gx>0$ for every column vector
$x\in\R^3\setminus\{0\}$. If $C$ is a cone invariant under
multiplication by positive scalars, we denote its positive projectivization by
\[
    \PP C:=C/\R_{>0}.
\]

\subsection{Gram and squared-edge coordinates}
\label{subsec:coordinate-models}

A \emph{labeled tetrahedron} is an ordered quadruple
\[
    \Delta=(p_0,p_1,p_2,p_3),
    \qquad p_i\in\R^3,
\]
whose vertices are affinely independent. We denote the set of labeled
nondegenerate tetrahedra by $\Tet_{\mathrm{lab}}$. Associated with
$\Delta$ is the \emph{edge matrix} based at $p_0$,
\[
    X(\Delta)
    :=
    \begin{bmatrix}
        p_1-p_0&p_2-p_0&p_3-p_0
    \end{bmatrix},
\]
and its \emph{Gram matrix}
\[
    G(\Delta):=X(\Delta)^{\tr}X(\Delta)\in\Sym_3^{>0}.
\]

Two labeled tetrahedra are called similar if they differ by a Euclidean similarity. Let
$\mathcal T_{\mathrm{lab}}$ denote the space of labeled similarity
classes.

\begin{proposition}
\label{prop:gram-model}
Let $\Delta,\Delta'\in\Tet_{\mathrm{lab}}$, with edge matrices
$X=X(\Delta)$ and $Y=X(\Delta')$. Then $\Delta$ and $\Delta'$ are
similar as labeled tetrahedra if and only if $G(\Delta')=c^2G(\Delta)$ for some $c>0$. Consequently,
\[
    \mathcal T_{\mathrm{lab}}
    \cong
    \PP\Sym_3^{>0} 
    \quad\text{ and }\quad 
    \det G(\Delta)=36\,\vol(\Delta)^2.
\]
\end{proposition}

\begin{proof}
Translation disappears when edge vectors are formed. Hence a
similarity sends $X$ to $cQX$ for some $c>0$ and $Q\in O(3)$, and
therefore sends $G=X^{\tr}X$ to $c^2G$. Conversely, if
$Y^{\tr}Y=c^2X^{\tr}X$, then
\[
    Q:=c^{-1}YX^{-1}
\]
satisfies $Q^{\tr}Q=I$, so $Q\in O(3)$ and $Y=cQX$. Every matrix in
$\Sym_3^{>0}$ has the form $X^{\tr}X$, which proves the projective
identification. Finally,
\[
    \vol(\Delta)=\frac16|\det X(\Delta)|
\]
and $\det G(\Delta)=\det X(\Delta)^2$.
\end{proof}

Let $\mathscr F = \{01,02,03,12,13,23\}$ be the set of six unoriented edges, in the displayed order. For a
labeled tetrahedron $\Delta$, put
\[
    d_{ij}(\Delta):=|p_i-p_j|^2,
    \qquad
    \mathbf{d}(\Delta)
    :=
    \begin{pmatrix}        d_{01}&d_{02}&d_{03}&d_{12}&d_{13}&d_{23}
    \end{pmatrix}^{\tr}.
\]

\begin{proposition}\label{prop:gram-edge-equivalence}
The linear map $\mathcal I:\Sym_3\longrightarrow\R^6$
defined by
\[
    \begin{pmatrix}
        g_{11}&g_{12}&g_{13}\\
        g_{12}&g_{22}&g_{23}\\
        g_{13}&g_{23}&g_{33}
    \end{pmatrix}
    \longmapsto
    \begin{pmatrix}
        g_{11}\\
        g_{22}\\
        g_{33}\\
        g_{11}+g_{22}-2g_{12}\\
        g_{11}+g_{33}-2g_{13}\\
        g_{22}+g_{33}-2g_{23}
    \end{pmatrix}
\]
is an isomorphism. If $G=G(\Delta)$, then $\mathcal I(G)=\mathbf{d}(\Delta)$.
\end{proposition}

\begin{proof}
The displayed identities follow by expanding the six squared-edge
lengths in terms of the three vectors $p_i-p_0$. Conversely, the
diagonal entries of $G$ are the first three coordinates, and the
remaining three identities determine $g_{12},g_{13},g_{23}$.
\end{proof}

For our applications in Sections \ref{sec:hyperbolic} and \ref{sec:elliptic} we record an explicit formula for $\mathcal I^{-1}$:  
\begin{equation}\label{eq:E_inverse}
    \mathcal I^{-1}(\mathbf{d}(\Delta))
    =
    \begin{pmatrix}
        d_{01}
        &
        \dfrac{d_{01}+d_{02}-d_{12}}{2}
        &
        \dfrac{d_{01}+d_{03}-d_{13}}{2}
        \\[3mm]
        \dfrac{d_{01}+d_{02}-d_{12}}{2}
        &
        d_{02}
        &
        \dfrac{d_{02}+d_{03}-d_{23}}{2}
        \\[3mm]
        \dfrac{d_{01}+d_{03}-d_{13}}{2}
        &
        \dfrac{d_{02}+d_{03}-d_{23}}{2}
        &
        d_{03}
    \end{pmatrix}.    
\end{equation}

\begin{definition}
\label{def:tet_cone_squared-edge}
The \emph{tetrahedral cone} in squared-edge coordinates is
\[
    \mathcal C
    :=
    \mathcal I(\Sym_3^{>0})
    \subset\R^6.
\]
Thus $\mathcal C$ consists precisely of the squared-edge vectors of
labeled nondegenerate tetrahedra. It is an open convex cone, and
\[
    \mathcal T_{\mathrm{lab}}\cong\PP\mathcal C.
\]
\end{definition}

Permuting the vertices gives an action of $S_4$ on labeled
 tetrahedra. We use the convention
\[
    \Delta^\sigma
    :=
    \bigl(p_{\sigma(0)},p_{\sigma(1)},p_{\sigma(2)},p_{\sigma(3)}\bigr),
    \qquad \sigma\in S_4.
\]
The space of unlabeled tetrahedral similarity classes is the finite
quotient
\[
    \mathcal T:=\mathcal T_{\mathrm{lab}}/S_4.
\]

\subsection{Relabelings and marked words}
\label{subsec:marked-words}

Let
\[
    \varepsilon_0:=0,
    \qquad
    \varepsilon_1:=e_1,
    \qquad
    \varepsilon_2:=e_2,
    \qquad
    \varepsilon_3:=e_3,
\]
where $e_1,e_2,e_3$ are the standard basis vectors of $\R^3$. For
$\sigma\in S_4$, define
\[
    R_\sigma
    :=
    \begin{bmatrix}
        \varepsilon_{\sigma(1)}-\varepsilon_{\sigma(0)}
        &
        \varepsilon_{\sigma(2)}-\varepsilon_{\sigma(0)}
        &
        \varepsilon_{\sigma(3)}-\varepsilon_{\sigma(0)}
    \end{bmatrix}.
\]
Then
\[
    X(\Delta^\sigma)=X(\Delta)R_\sigma,
    \qquad
    G(\Delta^\sigma)=R_\sigma^{\tr}G(\Delta)R_\sigma.
\]
Moreover, $R_\sigma\in\GL(3,\Z)$ and $|\det R_\sigma|=1$. In
squared-edge coordinates, relabeling permutes the six coordinates. We
denote the corresponding permutation matrix by $\Pi_\sigma$, so that
\[
    \mathbf{d}(\Delta^\sigma)=\Pi_\sigma \mathbf{d}(\Delta).
\]
More explicitly, with respect to the fixed edge order $(01,02,03,12,13,23)$, 
the permutation matrix $\Pi_\sigma$ is characterized by
\begin{equation}\label{eq:explicit_Pi}
    (\Pi_\sigma \mathbf{d})_{ij}
    =
    d_{\sigma(i)\sigma(j)},   
\end{equation}
where the indices on the right are regarded as an unordered edge.

\subsection{Directed bisections}
\label{subsec:Directed_bisections}
For distinct $i,j\in\{0,1,2,3\}$, the symbol
\[
    i\leftarrow j
\]
denotes the directed bisection that replaces $p_i$ by the midpoint
of $p_ip_j$ and leaves the other three vertices unchanged. Its
underlying unoriented edge is
\[
    \underline{i\leftarrow j}:=ij.
\]
The two children obtained by bisecting $ij$ are $i\leftarrow j$ and
$j\leftarrow i$. Let
\[
    \mathcal B
    :=
    \{i\leftarrow j \mid 0\leq i,j\leq3,\ i\neq j\}
\]
be the alphabet of directed bisections.

\begin{proposition}\label{prop:bisection-matrices}
For every $b=i\leftarrow j\in\mathcal B$, there is a matrix
$B_b\in\GL(3,\Q)$ such that
\[
    X(b\Delta)=X(\Delta)B_b,
    \qquad
    G(b\Delta)=B_b^{\tr}G(\Delta)B_b,
\]
for every labeled tetrahedron $\Delta$, and $|\det B_b|=\frac12$.
An explicit formula is obtained by setting
\[
    \varepsilon_k^{(b)}
    :=
    \begin{cases}
        \dfrac12(\varepsilon_i+\varepsilon_j),&k=i,\\[2mm]
        \varepsilon_k,&k\neq i,
    \end{cases}
\]
and defining
\[
    B_b
    :=
    \begin{bmatrix}
        \varepsilon_1^{(b)}-\varepsilon_0^{(b)}
        &
        \varepsilon_2^{(b)}-\varepsilon_0^{(b)}
        &
        \varepsilon_3^{(b)}-\varepsilon_0^{(b)}
    \end{bmatrix}.
\]
\end{proposition}

\begin{proof}
The new vertices have the displayed affine coordinates in the old
edge basis, which gives $X(b\Delta)=X(\Delta)B_b$. A child produced by
one bisection has half the volume of its parent, hence
$|\det B_b|=\frac12$.
\end{proof}

The action on squared-edge data is also linear. If
$b=i\leftarrow j$ and $\mathbf{d}'=\mathbf{d}(b\Delta)$, then
\[
    d'_{ij}=\frac14d_{ij},
\]
and, for $k\notin\{i,j\}$,
\[
    d'_{ik}
    =
    \frac12(d_{ik}+d_{jk})-\frac14d_{ij},
\]
while every edge not incident to $i$ is unchanged. We denote the
resulting matrix by $L_b\in\GL(6,\Q)$, so that
\[
    \mathbf{d}(b\Delta)=L_b\mathbf{d}(\Delta).
\]

These formulas determine the matrix $L_b$ row by row. More
explicitly, for an edge $e$, let $r_e$ denote the coordinate row
vector satisfying $r_e\mathbf{d}=d_e$. If $b=i\leftarrow j$, then, in our standard edge order,
\begin{equation}\label{eq:explicit_L}
  \operatorname{row}_e(L_{i\leftarrow j})
    =
    \begin{cases}
        \dfrac14 r_{ij},
            & e=ij,\\[2mm]
        -\dfrac14 r_{ij}
        +\dfrac12 r_{ik}
        +\dfrac12 r_{jk},
            & e=ik,\quad k\notin\{i,j\},\\[2mm]
        r_e,
            & i\notin e.
    \end{cases}    
\end{equation} 
Thus $L_{i\leftarrow j}$ is obtained simply by writing the six
updated squared-edge coordinates, in the order
$(01,02,03,12,13,23)$, as linear combinations of the original six
coordinates.

The two coordinate models are compatible:
\begin{equation}
\label{eq:bisection-model-compatibility}
    \mathcal I(B_b^{\tr}GB_b)
    =
    L_b\mathcal I(G),
    \qquad G\in\Sym_3.
\end{equation}

A \emph{directed bisection word} is a finite sequence
\[
    W=b_1b_2\cdots b_m,
    \qquad b_r\in\mathcal B,
\]
whose operations are performed from left to right. Set
\[
    B_W:=B_{b_1}B_{b_2}\cdots B_{b_m},
    \qquad
    L_W:=L_{b_m}\cdots L_{b_2}L_{b_1}.
\]
Then it follows from the definitions above that 
\[
    X(W\Delta)=X(\Delta)B_W,
    \qquad
    G(W\Delta)=B_W^{\tr}G(\Delta)B_W,
    \qquad
    \mathbf{d}(W\Delta)=L_W\mathbf{d}(\Delta),
\]
and $|\det B_W|=2^{-m}$.

A relabeling at the end of a word allows the same labeled block to be
iterated.

\begin{definition}\label{def:marked-word-matrix}
A \emph{marked word} is a pair
\[
    \omega=(W,\sigma),
    \qquad
    W\in\mathcal B^*,
    \quad
    \sigma\in S_4,
\]
where $\sigma$ is applied after all bisections in $W$. Its matrices in
the edge-matrix and squared-edge models are
\[
    P_\omega:=B_WR_\sigma,
    \qquad
    M_\omega:=\Pi_\sigma L_W,
\]
respectively. Thus
\[
    X\bigl((W\Delta)^\sigma\bigr)=X(\Delta)P_\omega
\quad
\text{ and }
\quad
    \mathbf{d}\bigl((W\Delta)^\sigma\bigr)=M_\omega \mathbf{d}(\Delta).
\]
For $P\in\GL(3,\R)$, denote its action on Gram matrices by 
\begin{equation}
\label{eq:C_P}
    \mathscr C_P:\Sym_3\longrightarrow\Sym_3,
    \qquad
    \mathscr C_P(G):=P^{\tr}GP.
\end{equation}
Then the two marked-word actions satisfy
\begin{equation}
\label{eq:marked-model-compatibility}
    \mathcal I\bigl(\mathscr C_{P_\omega}(G)\bigr)
    =
    M_\omega\mathcal I(G),
    \qquad G\in\Sym_3.
\end{equation}
\end{definition}

\subsection{LEB admissibility}
\label{subsec:admissibility}

The matrices above are defined for every tetrahedron. Their
interpretation as longest-edge bisections requires the selected edge
to be longest at the stage at which it is bisected.

Let $W=b_1\cdots b_m$ be a directed bisection word and let
$\mathbf{d}\in\mathcal C$. Define the successive squared-edge vectors by
\[
    \mathbf{d}^{(0)}:=\mathbf{d},
    \qquad
    \mathbf{d}^{(r)}
    :=
    L_{b_r}\mathbf{d}^{(r-1)}
    =
    L_{b_r}\cdots L_{b_1}\mathbf{d}
    \quad(1\leq r\leq m).
\]
We say that $W$ is \emph{LEB-admissible at $\mathbf{d}$} if
\begin{equation}
\label{eq:word-admissibility}
    d^{(r-1)}_{\underline{b_r}}
    =
    \max_{e\in\mathscr F}d^{(r-1)}_e
    \qquad
    (1\leq r\leq m).
\end{equation}
Thus each prescribed edge is longest immediately before it is
bisected. A marked word $\omega=(W,\sigma)$ is admissible at $\mathbf{d}$ when
its underlying bisection word $W$ is admissible there.

Accordingly, $P_\omega$ and $M_\omega$ define global matrix
transformations, but their interpretation as an LEB block is valid
only at shapes satisfying~\eqref{eq:word-admissibility}. 

\section{The hyperbolic construction}\label{sec:hyperbolic}

In this section we prove our main results by constructing one-parameter families
\[
    \{\Delta_t\mid t\in(0,1]\}
\]
of pairwise non-similar tetrahedra which are preserved projectively by a particular marked word. We show that the marked word takes $\Delta_t$ to a tetrahedron similar to $\Delta_{t/m}$, for an integer $m>1$, and that the normalized volume of $\Delta_t$ tends to zero with $t$. Applying the construction to $t\in(\frac1{m},1]$ yields uncountably many pairwise non-similar initial tetrahedra with degenerating branches.

\subsection{Method}
\label{subsec:hyperbolic-method}

Let $\omega$ be a marked word, let $P=P_\omega\in\GL(3,\R)$, and
consider the induced operator 
    $\mathscr C_P(G)=P^{\tr}GP$ from \eqref{eq:C_P}.
Thus, considering the quadratic form  $Q_G(x)=x^{\tr}Gx$, we obtain
\[
    Q_{\mathscr C_P(G)}(x)=Q_G(Px).
\]
In other words, $\mathscr C_P$ is the pullback action of $P$ on
quadratic forms.

Suppose that $P$ is diagonalizable, with eigenvalues
$\alpha,\beta,\overline\beta$, where $\alpha\in\R$, $\beta\in\C\setminus\R$ and
\[
    0<|\alpha|<|\beta|.
\]
Let $L\subset\R^3$ be the $\alpha$-eigendirection. On the quotient
$\R^3/L$, the map induced by $P$ is similar to $|\beta|$ times a
rotation. It therefore scales some positive-definite quadratic form
by $|\beta|^2$.
Indeed, write $\beta=a+ib$ and choose an eigenvector
$u+iv$ of the quotient action corresponding to $\beta$. Then
$u,v$ form a real basis of $\R^3/L$, and in this basis the quotient
action is represented by
\[
    \begin{pmatrix}a&b\\-b&a\end{pmatrix}
    =
    |\beta|
    \begin{pmatrix}\cos\theta&\sin\theta\\
                    -\sin\theta&\cos\theta
    \end{pmatrix}.
\]
Consequently, the quadratic form for which $u,v$ are orthonormal is
scaled by $|\beta|^2$.
Pulling this form back to $\R^3$ gives a
positive-semidefinite rank-two form $G_\parallel$ satisfying
\begin{equation}\label{eq:G_parallel_prop}
    \ker G_\parallel=L,
    \qquad
    \mathscr C_P(G_\parallel)=|\beta|^2G_\parallel.
\end{equation}

This form can also be obtained directly from the eigenvectors of
$P^{\tr}$. In general, if
\[
    P^{\tr}q_\gamma=\gamma q_\gamma,
    \qquad
    P^{\tr}q_\delta=\delta q_\delta,
\]
are left eigenvectors of $P$, then
\[
    \mathscr C_P\left(q_\gamma q_\delta^{\tr}+q_\delta q_\gamma^{\tr}\right)
    =
    \gamma\delta\left(q_\gamma q_\delta^{\tr}+q_\delta q_\gamma^{\tr}\right).
\]
Hence, choosing $q_\beta\in\C^3$ such that
$P^{\tr}q_\beta=\beta q_\beta$, we may take
$G_\parallel =     q_\beta\overline{q_\beta}^{\tr}
    +     \overline{q_\beta}q_\beta^{\tr}$.
 Then, for $x\in\R^3$ one sees that
\[
    x^{\tr}G_\parallel x
    =
    2\left|q_\beta^{\tr}x\right|^2,
\]
which makes both its positivity and its kernel transparent.

Now choose a real vector $q_\alpha$ such that
$P^{\tr}q_\alpha=\alpha q_\alpha$,
and set
$
    G_\perp:=q_\alpha q_\alpha^{\tr}.
$
Then $G_\perp$ is positive on $L$ and
\begin{equation}\label{eq:G_perp_prop}
    \mathscr C_P(G_\perp)=\alpha^2G_\perp.
\end{equation}
Writing
\[
    \lambda:=|\beta|^2,
    \qquad
    \mu:=\alpha^2,
\]
we have $0<\mu<\lambda$. Consequently, for every $t>0$,
\[
    G(t):=G_\parallel+tG_\perp
\]
is positive definite and
\[
    \mathscr C_P(G(t))
    =
    \lambda G\left(\frac{\mu}{\lambda}t\right).
\]
Thus the induced projective action contracts the parameter by
$
    \rho:=\frac{\mu}{\lambda}\in(0,1).
$
We note here that passing to squared-edge coordinates, with
\[
    \mathbf{d}(t):=\mathcal I(G(t))
         =\mathbf{d}_\parallel+t\mathbf{d}_\perp,
\]
gives
$M_\omega \mathbf{d}(t)=\lambda \mathbf{d}(\rho t)$.

Since $G_\parallel$ has rank two and $G_\perp$ is positive on its
kernel,
\[
    \det G(t)=ct
\]
for $c = \det(G_\parallel + G_\perp)>0$ (since $G_\parallel + G_\perp$ is positive definite). Hence $G(t)$ approaches a rank-two boundary
configuration as $t\to0$, and the normalized volume tends to zero.
To obtain an actual LEB branch, it remains only to find a suitable word $\omega$ and to verify that it is admissible at every step.

\subsection{Degeneration with tie-breaking}
\label{subsec:hyperbolic-main}
A brute-force search over marked words of length two, following the methodology above, found the word $\omega_{\hh,2}$ presented here.
The search also showed that, up to simultaneous relabeling of the vertices and a change of
the starting phase of the repeated block, it is the unique
length-two marked word for which that construction yields an
LEB-admissible family on some interval $0<t\le t_0$, which produces a degenerating branch. In particular, the example below and Korotov's example from \cite{Korotov} are one and the same. 

Let
\[
    W_{\hh,2}
    :=
    (0\leftarrow1)(1\leftarrow2),
    \qquad
    \sigma_{\hh,2}
    :=
    (1\,3\,2),
    \qquad
    \omega_{\hh,2}
    :=
    (W_{\hh,2},\sigma_{\hh,2}),
\]
where $\sigma_{\hh,2}$ fixes $0$. Thus the two bisections are
followed by the relabeling
\[
    0\mapsto0,
    \qquad
    1\mapsto3,
    \qquad
    3\mapsto2,
    \qquad
    2\mapsto1.
\]

\begin{lemma}\label{lem:constructing_G_t}
    The one-parameter family 
    \begin{equation}\label{eq:G_h(t)}
    G_{\hh,2}(t)
    :=
    \begin{pmatrix}
        2+t&t&\frac12+t\\
        t&t&t\\
        \frac12+t&t&1+t
    \end{pmatrix}, \qquad t\in(0,1]
\end{equation}
    of tetrahedra in the Gram-matrix model is a projectively $\omega_{\hh,2}$-invariant family of nondegenerate tetrahedra.  
\end{lemma}

\begin{proof}
As in Definition \ref{def:marked-word-matrix}, the associated matrix action of $\omega_{\hh,2}$ in the Gram model is
\[
    P_{\hh,2}
    =
    \begin{pmatrix}
        -\frac12&0&-\frac12\\
        0&\frac12&1\\
        1&0&0
    \end{pmatrix},
\]
and its characteristic polynomial is
\[
    \chi_{P_{\hh,2}}(z)
    =
    \frac14(2z-1)(2z^2+z+1).
\]
Indeed, $P_{\hh,2}$ has a real eigenvalue that is $\frac12$, while the other two eigenvalues are $\frac{-1\pm i\sqrt7}{4},
$
both of modulus $\frac{1}{\sqrt2}$. Thus, in the notation of the preceding
subsection,
\[
    \alpha=\frac12,
    \qquad
    |\beta|=\frac1{\sqrt2},
    \qquad
    \mu=\frac14,
    \qquad
    \lambda=\frac12.
\]

We now derive the relevant Gram eigenvectors. Let
$\mathscr C_{P_{\hh,2}}(G)
    =
    P_{\hh,2}^{\tr}GP_{\hh,2}$, as in \eqref{eq:C_P}.
The real eigenspace of $P_{\hh,2}$ is
$L= \spn \{ e_2\}$, namely 
    $P_{\hh,2} e_2=\frac12e_2$.
A symmetric form with kernel $L$ has the shape
\[
    G=
    \begin{pmatrix}
        a&0&c\\
        0&0&0\\
        c&0&b
    \end{pmatrix}.
\]
As in \eqref{eq:G_parallel_prop}, the matrix $G_\parallel$ should satisfy 
\begin{equation}\label{eq:G_parallel}
    \mathscr C_{P_{\hh,2}}(G_\parallel)=\frac12G_\parallel.
\end{equation}
Solving this equation with $G$ gives
$b=\frac a2$ and $c=\frac a4$. Choosing $a=2$ yields
\[
    G_\parallel
    =
    \begin{pmatrix}
        2&0&\frac12\\
        0&0&0\\
        \frac12&0&1
    \end{pmatrix}.
\]
For the transverse form, the vector
$q^{\tr}=
    \begin{pmatrix}
        1&1&1
    \end{pmatrix}$
satisfies
$P_{\hh,2}^{\tr}q=\frac12q$.
Therefore
\[
    G_\perp
    :=
    qq^{\tr}
    =
    \begin{pmatrix}
        1&1&1\\
        1&1&1\\
        1&1&1
    \end{pmatrix}
\]
satisfies
\begin{equation}\label{eq:G_perp}
\mathscr C_{P_{\hh,2}}(G_\perp)
    =
    \frac14G_\perp,   
\end{equation}     
    as in \eqref{eq:G_perp_prop}.
Consequently,
\begin{equation}\label{eq:G_h(t)_decomposition}
    G_{\hh,2}(t)
    =
    G_\parallel+tG_\perp
    =
    \begin{pmatrix}
        2+t&t&\frac12+t\\
        t&t&t\\
        \frac12+t&t&1+t
    \end{pmatrix},
\end{equation}
which by \eqref{eq:G_parallel} and \eqref{eq:G_perp} satisfies 
\begin{equation}\label{eq:returens_with_t/2}
\mathscr C_{P_{\hh,2}}(G_{\hh,2}(t)) = 
\mathscr C_{P_{\hh,2}}(G_\parallel) + t\mathscr C_{P_{\hh,2}}(G_\perp) = 
\frac12 G_\parallel + t\frac14 G_\perp =  \frac12 \left( G_{\hh,2}\left(\frac{1}{2}t\right)\right).
\end{equation}
In particular, the family $G_{\hh,2}(t)$ is projectively $\omega_{\hh,2}$-invariant. 
Since the leading principal minors of $G_{\hh,2}(t)$ are
\begin{equation}\label{eq:minors_of_G_h}
    2+t,
    \qquad
    2t,
    \qquad
    \frac{7t}{4},    
\end{equation}
we have $G_{\hh,2}(t)>0$ when $t>0$, proving the assertion. 
\end{proof}

Next, we choose vertex coordinates for the corresponding tetrahedra from the factorization of the Gram
matrix. Indeed, one can check that 
\[
    G_{\hh,2}(t)
    =
    X_{\hh,2}(t)^{\tr}X_{\hh,2}(t),
\]
where
\[
    X_{\hh,2}(t)
    =
    \begin{pmatrix}
        \frac12&0&1\\[1mm]
        \frac{\sqrt7}{2}&0&0\\[1mm]
        -\sqrt t&-\sqrt t&-\sqrt t
    \end{pmatrix}.
\]
Taking the columns of $X_{\hh,2}(t)$ as the edge vectors based at
$p_0(t)$ gives the convenient realization
\begin{equation*}\label{eq:vertices}
    p_0(t)=(0,0,\sqrt t),
    \qquad
    p_1=\left(\frac12,\frac{\sqrt7}{2},0\right),
    \qquad
    p_2=(0,0,0),
    \qquad
    p_3=(1,0,0).
\end{equation*}
Let
\[
    \Delta_t
    :=
    (p_0(t),p_1,p_2,p_3).
\]
Thus $t$ is the squared height of $p_0(t)$ above the face
$p_1p_2p_3$, and $t=0$ is the planar boundary configuration in which
$p_0=p_2$. 

The method above applies to any marked word $\omega$ whose matrix
$P_\omega$ satisfies the spectral hypotheses of
\S~\ref{subsec:hyperbolic-method}. For our application, the key property of such a family is admissibility with respect to LEB and the opposite-edge tie-breaking rule, at every step. In our brute-force search, we used the squared-edge lengths model to check this admissibility. 

\begin{lemma}\label{lem:admissibility}
    For every $t\in(0,1]$, the bisections induced by $W_{\hh,2} = (0\leftarrow1)(1\leftarrow2)$ on $\Delta_t$ are LEB admissible and obey the opposite-edge tie-breaking rule. 
\end{lemma}

\begin{proof}
To verify admissibility, we first describe the family $\Delta_t$ and the above action in the squared-edge lengths model. In view of \eqref{eq:G_h(t)_decomposition}, and by Proposition \ref{prop:gram-edge-equivalence}, the corresponding vector is
\begin{equation}\label{eq:d_h(t)}
    \mathbf{d}_{\hh,2}(t)
    =
    \begin{pmatrix}
        2+t\\
        t\\
        1+t\\
        2\\
        2\\
        1
    \end{pmatrix}
    =
    \underbrace{
    \begin{pmatrix}
        2\\0\\1\\2\\2\\1
    \end{pmatrix}}_{\mathbf{d}_\parallel}
    +
    t
    \underbrace{
    \begin{pmatrix}
        1\\1\\1\\0\\0\\0
    \end{pmatrix}}_{\mathbf{d}_\perp}.    
\end{equation}

Recall that the coordinate order is
$(01,02,03,12,13,23)$. By \eqref{eq:d_h(t)}, for every $t>0$, the unique longest edge of $\Delta_t$ is $01$, so the first letter
is LEB admissible, and the tie-breaking rule is irrelevant.

The letter $0\leftarrow1$ is applied first. Using \eqref{eq:explicit_L}, it has a squared-edge matrix
\[
    L_{0\leftarrow1}
    =
    \begin{pmatrix}
        \frac14&0&0&0&0&0\\
        -\frac14&\frac12&0&\frac12&0&0\\
        -\frac14&0&\frac12&0&\frac12&0\\
        0&0&0&1&0&0\\
        0&0&0&0&1&0\\
        0&0&0&0&0&1
    \end{pmatrix},
\]
and after applying it, we obtain
\[
   \mathbf{d}^{(1)} = L_{0\leftarrow1}\mathbf{d}_{\hh,2}(t)
    =
    \left(
        \frac12+\frac t4,
        \frac12+\frac t4,
        1+\frac t4,
        2,
        2,
        1
    \right)^{\tr}.
\]
For $0<t\le 1$, the longest edges are exactly $12$ and $13$, both of squared length $2$. The edge opposite to $12$ is $03$, while the edge opposite to $13$
is $02$, and their squared lengths satisfy
\[
    d^{(1)}_{03}
    =
    1+\frac t4
    >
    \frac12+\frac t4
    =
    d^{(1)}_{02},
\]
for every $t\in (0,1]$. Hence the opposite-edge rule selects $12$, exactly as prescribed by
the second letter $1\leftarrow2$, and the proof is complete.
\end{proof}

\begin{lemma}\label{lem:non-similar}
    The tetrahedra $\{\Delta_t\}_{t\in(0,1]}$ are pairwise non-similar. 
\end{lemma}

\begin{proof}
    By \eqref{eq:d_h(t)}, for every $0<t\le 1$, the shortest squared-edge of $\Delta_t$ is $t$, and
    the longest is $2+t$. Hence
\[
    \frac{\min_{i<j}d_{ij}}{\max_{i<j}d_{ij}}
    =
    \frac{t}{2+t} = 1-\frac{2}{2+t}.
\]
This quantity is invariant under similarities and relabelings, and
is strictly increasing in $t$, which proves the lemma. 
\end{proof}

\begin{theorem}
\label{thm:hyperbolic-family}
For every $t_0\in(0,1]$, repeated application of the marked word $\omega_{\hh,2}$ produces an infinite LEB branch that obeys the opposite-edge
tie-breaking rule. After $n$ blocks, the resulting tetrahedron is
similar to $\Delta_{2^{-n}t_0}$, and the entire branch degenerates.
\end{theorem}

\begin{proof}
Set $t_n=2^{-n}t_0$. By
Lemma~\ref{lem:constructing_G_t}, one application of the marked word
$\omega_{\hh,2}$ takes $\Delta_{t_n}$ to a tetrahedron similar to
$\Delta_{t_{n+1}}$. Since $t_n\in(0,1]$ for every $n$,
Lemma~\ref{lem:admissibility} shows that every such block is
LEB-admissible and obeys the opposite-edge tie-breaking rule. Hence
the block can be iterated indefinitely.
By \eqref{eq:minors_of_G_h}
\[
    \vol(\Delta_t)
    =
    \frac16\sqrt{\det G_{\hh,2}(t)}
    =
    \frac{\sqrt{7t}}{12}.
\]
and so by \eqref{eq:d_h(t)}, and \eqref{eq:degeneration_parameter}, we have
\[
    \vartheta(\Delta_t)
    =
    \frac{\sqrt{7t}}{12(2+t)^{3/2}}
    \longrightarrow 0
    \qquad (t\to0).
\]
The only intermediate tetrahedron in each block is similar to
$(0\leftarrow1)\Delta_t$. Its volume is half that of $\Delta_t$, and
the computation in Lemma~\ref{lem:admissibility} shows that its
squared diameter is $2$. Therefore
\[
    \vartheta\bigl((0\leftarrow1)\Delta_t\bigr)
    =
    \frac{\sqrt{7t}}{48\sqrt2}
    \longrightarrow 0
    \qquad (t\to0).
\]
Thus both the marked and the intermediate subsequences degenerate,
and hence so does the entire branch.
\end{proof}

\begin{proof}[Proof of Theorem~\ref{thm:main-degeneration}]
By Lemma~\ref{lem:non-similar}, the uncountable family
\[
    \left\{\Delta_t\mid t\in\left(\tfrac12,1\right]\right\}
\]
consists of pairwise non-similar tetrahedra. By
Theorem~\ref{thm:hyperbolic-family}, each of these tetrahedra has an
infinite degenerating LEB branch that obeys the opposite-edge
tie-breaking rule. Since the sets of numbers $\{2^{-n}t_1\}_{n\in\N}$ and $\{2^{-n}t_2\}_{n\in\N}$ are disjoint for $t_1\neq t_2$, $t_1,t_2\in\left(\tfrac12,1\right]$, every two such branches are disjoint. 
\end{proof}

\subsection{Degeneration without tie-breaking}
\label{subsec:strict-hyperbolic}

We now prove Theorem~\ref{thm:main-degeneration2}. A word will be
called \emph{strictly LEB-admissible} if, immediately before each
bisection, the prescribed edge is the unique longest edge. Thus a
strictly admissible branch is independent of the rule used to resolve
ties. This condition concerns the chosen branch, rather than all
branches of the refinement tree.

For $0<t\leq1$, let $\Delta^{(32)}_t$ have vertices
\[
\begin{aligned}
 p_0&=(0,0,0),&
 p_1&=(40,8\sqrt7,-2\sqrt t),\\
 p_2&=(32,0,15\sqrt t),&
 p_3&=(8,8\sqrt7,6\sqrt t).
\end{aligned}
\]
Its Gram matrix and squared-edge vector are
\begin{equation}\label{eq:strict-hyperbolic-data}
\begin{aligned}
 G_{\hh,32}(t)&=256H_{\hh,32}+tq_{\hh,32}q_{\hh,32}^{\tr},\\
 H_{\hh,32}&=
 \begin{pmatrix}8&5&3\\5&4&1\\3&1&2\end{pmatrix},
 &q_{\hh,32}&=\begin{pmatrix}-2\\15\\6\end{pmatrix},\\
 \mathbf{d}_{\hh,32}(t)&=
 \begin{pmatrix}
 2048+4t\\1024+225t\\512+36t\\
 512+289t\\1024+64t\\1024+81t
 \end{pmatrix}.
\end{aligned}
\end{equation}
Here $H_{\hh,32}$ is positive semidefinite of rank two, and
\[
 \det G_{\hh,32}(t)=256^2\cdot23^2\cdot7t>0.
\]
Consider the length-$32$ word
\begin{equation}\label{eq:strict-hyperbolic-word}
\begin{aligned}
 W_{\hh,32}:={}&
 (0\leftarrow1)(2\leftarrow3)(3\leftarrow1)(2\leftarrow1)
 (1\leftarrow0)(2\leftarrow0)(0\leftarrow3)(3\leftarrow1)\\
 & (1\leftarrow0)(2\leftarrow3)(3\leftarrow0)(1\leftarrow0)
 (0\leftarrow2)(1\leftarrow2)(2\leftarrow3)(0\leftarrow3)\\
 & (3\leftarrow1)(2\leftarrow0)(0\leftarrow1)(2\leftarrow1)
 (1\leftarrow3)(2\leftarrow3)(3\leftarrow0)(1\leftarrow0)\\
 & (0\leftarrow2)(3\leftarrow1)(1\leftarrow2)(0\leftarrow2)
 (2\leftarrow3)(0\leftarrow3)(3\leftarrow1)(1\leftarrow2),
\end{aligned}
\end{equation}
and the marked word
\[
 \omega_{\hh,32}:=(W_{\hh,32},\sigma_{\hh,32}),
 \qquad \sigma_{\hh,32}:=(0\,2)(1\,3).
\]
Thus the final relabeling is
$(p_0,p_1,p_2,p_3)\mapsto(p_2,p_3,p_0,p_1)$.

\begin{lemma}\label{lem:strict-hyperbolic-return}
For every $0<t\leq1$, the word $\omega_{\hh,32}$ is strictly
LEB-admissible at $\Delta^{(32)}_t$, and
\[
 \mathscr C_{P_{\omega_{\hh,32}}}(G_{\hh,32}(t))
 =2^{-20}G_{\hh,32}(t/16).
\]
\end{lemma}

\begin{proof}
Using Proposition~\ref{prop:bisection-matrices} and
Definition~\ref{def:marked-word-matrix}, we obtain
\[
 P_{\omega_{\hh,32}}=
 \frac1{4096}
 \begin{pmatrix}-2&-3&0\\2&1&2\\-6&-1&-4\end{pmatrix}.
\]
Direct multiplication gives
\[
 P_{\omega_{\hh,32}}^{\tr}H_{\hh,32}P_{\omega_{\hh,32}}
 =2^{-20}H_{\hh,32},
 \qquad
 P_{\omega_{\hh,32}}^{\tr}q_{\hh,32}=2^{-12}q_{\hh,32},
\]
which proves the return identity. Also,
$\det P_{\omega_{\hh,32}}=2^{-32}$, as required for $32$ bisections.

It remains to verify strict admissibility. Write
$W_{\hh,32}=b_1\cdots b_{32}$, and let
$\mathbf{d}^{(k)}(t)=L_{b_k}\cdots L_{b_1}\mathbf{d}_{\hh,32}(t)$, with
$\mathbf{d}^{(0)}(t)=\mathbf{d}_{\hh,32}(t)$. Put
\[
 \delta_k(t):=
 d^{(k-1)}_{\underline{b_k}}(t)
 -\max_{e\in\mathscr F\setminus\{\underline{b_k}\}}
 d^{(k-1)}_e(t).
\]
The midpoint formulas~\eqref{eq:explicit_L} give the following
endpoint data, where
$\alpha_k=2^{r_k}\delta_k(0)$ and
$\beta_k=2^{r_k}\delta_k(1)$:
\[
\begin{array}{r|r|r|r@{\qquad}r|r|r|r}
 k&r_k&\alpha_k&\beta_k&k&r_k&\alpha_k&\beta_k\\\hline
 1&0&1024&803&17&12&4096&3808\\
 2&0&0&17&18&12&0&24\\
 3&2&2048&1679&19&14&8192&8003\\
 4&2&0&621&20&14&0&253\\
 5&4&4096&3856&21&16&16384&16428\\
 6&4&0&297&22&16&0&357\\
 7&6&0&560&23&18&0&1792\\
 8&6&0&528&24&18&0&799\\
 9&6&2048&1679&25&18&8192&7933\\
 10&6&0&117&26&20&0&140\\
 11&8&4096&3856&27&20&16384&16163\\
 12&8&0&207&28&20&0&161\\
 13&8&2048&2093&29&22&32768&33020\\
 14&10&0&576&30&22&0&765\\
 15&10&0&275&31&24&0&528\\
 16&10&0&112&32&24&0&135
\end{array}
\]
Every comparison defining $\delta_k$ is affine in $t$. Therefore
\[
 \delta_k(t)
 \geq(1-t)\delta_k(0)+t\delta_k(1)
 =2^{-r_k}\bigl((1-t)\alpha_k+t\beta_k\bigr)>0
 \qquad(0<t\leq1).
\]
This proves strict admissibility throughout the parameter interval.
\end{proof}

\begin{theorem}\label{thm:strict-hyperbolic-family}
For every $t_0\in(0,1]$, repeated application of $\omega_{\hh,32}$
produces a strictly admissible degenerating LEB branch. After $n$
marked blocks, its tetrahedron is similar to
$\Delta^{(32)}_{16^{-n}t_0}$.
\end{theorem}

\begin{proof}
The return and admissibility assertions follow by iterating
Lemma~\ref{lem:strict-hyperbolic-return}. Moreover,
\[
 \vartheta(\Delta^{(32)}_t)
 =\frac{256\cdot23\sqrt{7t}}
 {6(2048+4t)^{3/2}}
 \longrightarrow0
 \qquad(t\to0).
\]
This proves degeneration at the block boundaries. For any fixed
prefix $V$ of $W_{\hh,32}$, its Gram matrix is
$B_V^{\tr}G_{\hh,32}(t)B_V$. As $t\to0$, this converges to
$256B_V^{\tr}H_{\hh,32}B_V$, which has rank two because $B_V$ is
invertible. Hence its diameter has a positive limit, while its
volume tends to zero. Applying this argument to the finitely many
prefixes proves degeneration of the entire branch.
\end{proof}

\begin{proof}[Proof of Theorem~\ref{thm:main-degeneration2}]
For $0<t\leq1$, the shortest and longest squared-edge lengths of
$\Delta^{(32)}_t$ are $512+36t$ and $2048+4t$, respectively. Their
ratio
\[
 \frac{512+36t}{2048+4t}
\]
is strictly increasing in $t$ and is invariant under similarities
and relabelings. Thus the family consists of pairwise non-similar
tetrahedra. Theorem~\ref{thm:strict-hyperbolic-family} supplies the
required branches, with no use of a tie-breaking rule.
\end{proof}

\subsection{The higher-dimensional case}
\label{subsec:open-higher-dimensional}

We prove Theorem~\ref{thm:main-degeneration3} in the stronger form
of an open-set statement, and then prove
Corollary~\ref{cor:higher-dimensional-fsc}.
The Gram-matrix,
squared-edge, and marked-word conventions of
Section~\ref{sec:setup} extend to a labeled $d$-simplex by using
$d\times d$ Gram matrices and the $\binom{d+1}{2}$ unordered
edges. We denote the corresponding linear isomorphism by
$\mathcal I_d$, and recall the notation $\vartheta_d(\Sigma)$ from \eqref{eq:degeneration_parameter}. 
For symmetric matrices, $A\preceq B$ means that $B-A$ is positive
semidefinite, and $A\prec B$ means that $B-A$ is positive definite.

For $0<t\leq1$, let $\Sigma_t$ be the $4$-simplex with vertices
\[
\begin{aligned}
 p_0&=(0,0,0,0),&p_1&=(26,0,0,2\sqrt t),\\
 p_2&=(16,18,0,\sqrt t),&p_3&=(10,2,8,2\sqrt t),\\
 p_4&=(11,7,-8,\sqrt t).
\end{aligned}
\]
Set
\[
 A_{\hh,9}=
 \begin{pmatrix}26&16&10&11\\0&18&2&7\\0&0&8&-8\end{pmatrix},
 \qquad
 v_{\hh,9}=\begin{pmatrix}-1\\-1\\2\\2\end{pmatrix},
 \qquad
 q_{\hh,9}=\begin{pmatrix}2\\1\\2\\1\end{pmatrix}.
\]
Then $A_{\hh,9}v_{\hh,9}=0$, $q_{\hh,9}^{\tr}v_{\hh,9}=3$, and
\begin{equation}\label{eq:open-four-gram}
 G_{\hh,9}(t):=G(\Sigma_t)=H_{\hh,9}+tK_{\hh,9},
 \qquad H_{\hh,9}:=A_{\hh,9}^{\tr}A_{\hh,9},
 \qquad K_{\hh,9}:=q_{\hh,9}q_{\hh,9}^{\tr}.
\end{equation}
The edge matrix satisfies $\det X(\Sigma_t)=5616\sqrt t$, so these
simplices are nondegenerate for $t>0$.

Consider the marked word $\omega_{\hh,9}=(W_{\hh,9},\sigma_{\hh,9})$, where 
\begin{equation}\label{eq:open-four-word}
\begin{aligned}
 W_{\hh,9}:={}&(0\leftarrow1)(1\leftarrow2)(2\leftarrow3)
          (3\leftarrow4)(4\leftarrow1)\\
        & (1\leftarrow0)(0\leftarrow2)(2\leftarrow4)
          (4\leftarrow3),
 \qquad \sigma_{\hh,9}:=(0\,3).
\end{aligned}
\end{equation}
Its edge-matrix return is
\begin{equation}\label{eq:open-four-return-matrix}
 P_{\hh,9}:=P_{\omega_{\hh,9}}
 =\frac14I-\frac18v_{\hh,9}q_{\hh,9}^{\tr}
 =\frac18
 \begin{pmatrix}
 4&1&2&1\\2&3&2&1\\-4&-2&-2&-2\\-4&-2&-4&0
 \end{pmatrix}.
\end{equation}
In particular,
\[
 A_{\hh,9}P_{\hh,9}=\frac14A_{\hh,9},
 \qquad q_{\hh,9}^{\tr}P_{\hh,9}=-\frac18q_{\hh,9}^{\tr},
 \qquad \det P_{\hh,9}=-2^{-9},
\]
and hence
\begin{equation}\label{eq:open-four-return}
 \mathscr C_{P_{\hh,9}}(G_{\hh,9}(t))=\frac1{16}G_{\hh,9}(t/4).
\end{equation}

\begin{lemma}\label{lem:open-four-gaps}
The word $W_{\hh,9}$ is strictly LEB-admissible at $\Sigma_t$ for every
$0<t\leq1$. Its edge comparisons remain strict at the rank-three
limiting configuration $t=0$.
\end{lemma}

\begin{proof}
In the edge order
\[
 01,02,03,04,12,13,14,23,24,34,
\]
the initial squared-edge vector is
\[
 \mathcal I_4(G_{\hh,9}(t))=
 \bigl(676+4t,580+t,168+4t,234+t,424+t,
       324,338+t,356+t,210,282+t\bigr)^{\tr}.
\]
The midpoint formulas~\eqref{eq:explicit_L}, in dimension four,
give the following largest competitors and squared-length gaps
throughout $0\leq t\leq1$:
\[
\begin{array}{c|c|c|c}
 k&b_k&\text{largest competing edge}&\text{gap}\\\hline
 1&0\leftarrow1&02&96+3t\\
 2&1\leftarrow2&23&68\\
 3&2\leftarrow3&02&23+t\\
 4&3\leftarrow4&13&48+3t/4\\
 5&4\leftarrow1&24&11\\
 6&1\leftarrow0&13&29/2+t/4\\
 7&0\leftarrow2&04&27+3t/16\\
 8&2\leftarrow4&12&59/4\\
 9&4\leftarrow3&04&9/2+t/16
\end{array}
\]
All entries are at least $9/2$.
\end{proof}

The strict inequalities at the limiting configuration allow an
open set of perturbations. The following gives explicit constants.

\begin{theorem}
\label{thm:open-four-dimensional}
Let
\[
 G_*:=G_{\hh,9}(1)=
 \begin{pmatrix}
 680&418&264&288\\418&581&198&303\\
 264&198&172&62\\288&303&62&235
 \end{pmatrix},
\]
and define
\begin{equation}\label{eq:open-four-neighborhood}
 \mathcal U_{\hh,9}:=
 \left\{G\in\Sym_4^{>0} \:\Big{|}\:
 \|G-G_*\|_{\max}<\frac1{100}\right\},
 \qquad \|D\|_{\max}:=\max_{i,j}|D_{ij}|.
\end{equation}
Every simplex with Gram matrix in $\mathcal U_{\hh,9}$ admits infinite
repetition of $\omega_{\hh,9}$. The resulting branch is strictly
LEB-admissible and degenerates. Consequently, the set of
$4$-dimensional similarity classes having a degenerating branch
contains a nonempty open set.
\end{theorem}

\begin{proof}
Put
\[
 Q_{\hh,9}:=4P_{\hh,9},
 \qquad F_{\hh,9}:=\frac13v_{\hh,9}q_{\hh,9}^{\tr},
 \qquad E_{\hh,9}:=I-F_{\hh,9}.
\]
Since $q_{\hh,9}^{\tr}v_{\hh,9}=3$, the matrices $E_{\hh,9},F_{\hh,9}$ are complementary
projections, of ranks three and one, respectively. Therefore
\begin{equation}\label{eq:open-four-normalized-powers}
 Q_{\hh,9}=E_{\hh,9}-\frac12F_{\hh,9},
 \qquad
 Q_{\hh,9}^n=E_{\hh,9}+\left(-\frac12\right)^nF_{\hh,9}.
\end{equation}
Let $\|\cdot\|_1$ denote the maximum absolute column sum. Since
$\|F_{\hh,9}\|_1=4$, the identity
\[
 Q_{\hh,9}^n=I+\bigl((-1/2)^n-1\bigr)F_{\hh,9}
\]
gives $\|Q_{\hh,9}^n\|_1\leq7$ for all $n\geq0$.

For $G\in\mathcal U_{\hh,9}$, the normalized Gram matrix at the start
of block $n$ is
\[
 \widehat G_n:=(Q_{\hh,9}^n)^{\tr}GQ_{\hh,9}^n.
\]
The reference matrix at that stage is
$(Q_{\hh,9}^n)^{\tr}G_*Q_{\hh,9}^n=H_{\hh,9}+4^{-n}K_{\hh,9}$, and
\begin{equation}\label{eq:open-four-error}
 \|\widehat G_n-(H_{\hh,9}+4^{-n}K_{\hh,9})\|_{\max}
 \leq49\|G-G_*\|_{\max}<\frac{49}{100}.
\end{equation}
Every vertex produced by a prefix of $W_{\hh,9}$ is a convex combination
of the five vertices at the start of the block. Thus every edge has
a coefficient vector $a$ in the starting edge basis with
$\sum_i|a_i|\leq2$. Under a Gram perturbation $D$, its squared
length changes by at most $4\|D\|_{\max}$, and a comparison
between two squared lengths changes by at most
$8\|D\|_{\max}$. Equivalently, a perturbation of size $\eta$ at
the initial marked Gram matrix changes every comparison at any
normalized block by at most
\begin{equation}\label{eq:four-uniform-comparison-error}
 8\cdot49\eta=392\eta.
\end{equation}
By Lemma~\ref{lem:open-four-gaps} and
\eqref{eq:open-four-error}, every prescribed gap for $G$ is greater
than
\[
 \frac92-\frac{392}{100}=\frac{29}{50}>0.
\]
This proves strict admissibility of every block.

By~\eqref{eq:open-four-normalized-powers},
\[
 \widehat G_n\longrightarrow E_{\hh,9}^{\tr}GE_{\hh,9}.
\]
The limit has rank three and positive diameter. Also
$|\det Q_{\hh,9}|=1/2$, so the normalized $4$-volume decreases by the
factor $2^{-n}$. Thus the branch degenerates at the block
boundaries. For every fixed prefix $V$ of $W_{\hh,9}$, the intermediate
normalized Gram matrices converge to
$B_V^{\tr}E_{\hh,9}^{\tr}GE_{\hh,9}B_V$. This again has rank three, since
$B_V$ is invertible, while the corresponding normalized volumes
tend to zero. The finitely many prefixes therefore give degeneration
of the entire branch.

Finally, positive projectivization is an open map. Hence the image
of $\mathcal U_{\hh,9}$ is a nonempty open subset of shape space. In
particular, this proves Theorem~\ref{thm:main-degeneration3}.
\end{proof}

\begin{lemma}\label{lem:lift-face-branch}
Let $\Delta$ be a $k$-dimensional face of a nondegenerate
$d$-simplex $\Sigma$, where $2\leq k<d$. Every degenerating
LEB branch of $\Delta$ can be lifted to a degenerating LEB branch
of $\Sigma$, allowing arbitrary choices between tied longest
edges.
\end{lemma}

\begin{proof}
Suppose that the current ambient simplex contains the current
face $\Delta$, and let $e$ be the next edge prescribed by its
branch. If $e$ is not the longest edge in the ambient simplex, choose a
longer ambient longest edge $f$, bisect it, and retain the child
containing all vertices of $\Delta$. Such a child exists because
$f$ cannot be an edge of $\Delta$.

This waiting stage is finite. Otherwise an infinite LEB branch
would retain a fixed face of positive diameter, contradicting
Kearfott's diameter-convergence theorem
\cite[Theorem~3.1]{Kearfott1978}. Once $e$ is longest in the ambient
simplex, perform its prescribed directed bisection. Repetition
lifts the whole face branch.

If $\Delta=[p_0,\ldots,p_k]$ is the current distinguished face of
$\Sigma=[p_0,\ldots,p_d]$, the exterior-product volume formula gives
\[
 d!\,\vol_d(\Sigma)
 \leq k!\,\vol_k(\Delta)\prod_{j=k+1}^d|p_j-p_0|
 \leq k!\,\vol_k(\Delta)\diam(\Sigma)^{d-k}.
\]
Consequently,
\begin{equation}\label{eq:face-quality-bound}
 \vartheta_d(\Sigma)
 \leq\frac{k!}{d!}
      \frac{\vol_k(\Delta)}{\diam(\Sigma)^k}
 \leq\frac{k!}{d!}\vartheta_k(\Delta).
\end{equation}
The index of the distinguished face along its branch tends to
infinity, since each waiting stage is finite. Thus the displayed
bound proves degeneration, including during the waiting stages.
\end{proof}

\begin{lemma}\label{lem:full-tree-tie-free}
In every dimension $d\geq2$, the initial Gram matrices having
some descendant with two equal edge lengths are contained in a
countable union of proper rational hyperplanes. In particular,
their complement is dense and has full Lebesgue measure in
$\Sym_d^{>0}$. If the initial squared-edge coordinates are linearly
independent over $\Q$, every descendant has pairwise distinct
edge lengths, even under arbitrary directed bisections.
\end{lemma}

\begin{proof}
For every finite directed word $V$, its squared-edge matrix is
$L_V\in\GL(\binom{d+1}{2},\Q)$. Equality between distinct
edges $e,f$ of the descendant is the equation
\[
 (r_e-r_f)L_V\mathbf{d}=0.
\]
The coefficient row is nonzero because $L_V$ is invertible.
Pulling back by $\mathcal I_d$ therefore gives a proper rational
hyperplane in Gram space. There are only countably many words
and edge pairs. The last assertion follows because a nonzero
rational linear form cannot vanish on a vector with
$\Q$-linearly independent coordinates.
\end{proof}

\begin{corollary}\label{cor:open-all-dimensions}
For every $d\geq4$, the set of $d$-dimensional similarity classes
having a degenerating LEB branch contains a nonempty open set.
Within this open set, a dense full-measure subset consists of
simplices whose entire refinement trees have pairwise distinct
edge lengths at every node.
\end{corollary}

\begin{proof}
For $d>4$, prescribe a labeled $4$-face with Gram matrix in
$\mathcal U_{\hh,9}$ and allow the remaining vertices to vary subject
only to nondegeneracy. This is a nonempty open condition. Apply
Theorem~\ref{thm:open-four-dimensional} and
Lemma~\ref{lem:lift-face-branch}. The assertion about ties follows
from Lemma~\ref{lem:full-tree-tie-free}, and passes to projective
shape space and its finite unlabeled quotient.
\end{proof}

An open set alone does not imply the probabilistic statement in
Corollary~\ref{cor:higher-dimensional-fsc}: a fixed Gaussian
$4$-face in high ambient dimension concentrates near the regular
shape. We therefore record a conical family of degenerating
simplices accumulating at that shape.

\begin{lemma}\label{lem:regular-four-cone}
In the edge order used above, put
\[
 \beta:=(-3,-7,1,-7,3,5,5,3,7,5)^{\tr},
 \qquad \mathbf1:=(1,\ldots,1)^{\tr}\in\R^{10}.
\]
Every squared-edge vector in 
\begin{equation}\label{eq:regular-four-cone}
 \mathcal K:=
 \left\{\mathbf1+\varepsilon(\beta+h) \;\Big|\;
 0<\varepsilon\leq10^{-7},\quad
 h\in\R^{10},\quad\|h\|_\infty<\frac12\right\}
\end{equation}
represents a nondegenerate $4$-simplex with a strictly admissible
degenerating LEB branch. The set $\mathcal K$ is open in
squared-edge coordinates.
\end{lemma}

\begin{proof}
Let $V=c_1\cdots c_{100}$ be the word whose letters are listed
below, read from left to right and then from top to bottom:
\[
\begin{array}{cccccccccc}
 4\leftarrow2&1\leftarrow3&2\leftarrow3&3\leftarrow0&0\leftarrow1&1\leftarrow4&3\leftarrow4&0\leftarrow4&4\leftarrow2&2\leftarrow0\\
 0\leftarrow4&1\leftarrow4&4\leftarrow3&1\leftarrow3&3\leftarrow2&2\leftarrow0&0\leftarrow3&3\leftarrow4&1\leftarrow2&2\leftarrow4\\
 1\leftarrow4&0\leftarrow4&4\leftarrow3&3\leftarrow2&4\leftarrow2&2\leftarrow0&1\leftarrow3&3\leftarrow0&0\leftarrow4&3\leftarrow4\\
 4\leftarrow2&2\leftarrow1&1\leftarrow0&2\leftarrow0&0\leftarrow4&1\leftarrow4&4\leftarrow3&3\leftarrow2&2\leftarrow4&3\leftarrow0\\
 0\leftarrow4&1\leftarrow4&4\leftarrow3&1\leftarrow3&0\leftarrow2&3\leftarrow4&2\leftarrow4&0\leftarrow4&4\leftarrow1&1\leftarrow3\\
 3\leftarrow2&2\leftarrow0&1\leftarrow0&3\leftarrow0&4\leftarrow3&2\leftarrow0&0\leftarrow1&2\leftarrow1&3\leftarrow4&4\leftarrow1\\
 0\leftarrow1&2\leftarrow1&1\leftarrow3&3\leftarrow0&4\leftarrow0&0\leftarrow1&2\leftarrow1&4\leftarrow1&1\leftarrow3&3\leftarrow2\\
 1\leftarrow0&0\leftarrow2&4\leftarrow2&1\leftarrow2&2\leftarrow3&0\leftarrow3&4\leftarrow3&3\leftarrow1&4\leftarrow1&2\leftarrow0\\
 1\leftarrow3&0\leftarrow4&4\leftarrow3&2\leftarrow3&1\leftarrow3&3\leftarrow0&2\leftarrow0&0\leftarrow4&4\leftarrow2&2\leftarrow1\\
 3\leftarrow0&0\leftarrow1&3\leftarrow1&4\leftarrow1&1\leftarrow2&0\leftarrow2&2\leftarrow4&3\leftarrow1&0\leftarrow4&1\leftarrow4
\end{array}
\]
There are no intermediate relabelings. After $V$, apply the
permutation $\sigma_V$ given by
\[
 (p_0,p_1,p_2,p_3,p_4)\mapsto(p_4,p_2,p_3,p_0,p_1),
\]
and thereafter repeat $\omega_{\hh,9}$.

We give finite rational certificates for both parts of this
branch. For a selected-edge comparison in the prefix, let
\[
 r=(r_{\underline{c_j}}-r_e)L_{c_{j-1}}\cdots L_{c_1},
 \qquad A_r:=r\mathbf1,
 \qquad B_r:=r\beta,
 \qquad N_r:=\sum_{i=1}^{10}|r_i|.
\]
There are $900$ such rows. Successive use of the midpoint
formulas gives
\begin{equation}\label{eq:regular-prefix-certificate}
\begin{gathered}
 A_r\geq0,\\
 A_r=0\ \Longrightarrow\ B_r\geq N_r,\\
 A_r>0\ \Longrightarrow\
 A_r+\frac2{353}\left(B_r-\frac12N_r\right)\geq0.
\end{gathered}
\end{equation}
In particular, for $\|h\|_\infty<1/2$,
\[
 r\bigl(\mathbf1+\varepsilon(\beta+h)\bigr)
 \geq A_r+\varepsilon(B_r-N_r/2)>0
 \qquad(0<\varepsilon<2/353).
\]
Thus the prefix is strictly admissible on~\eqref{eq:regular-four-cone}.

The regular initial Gram matrix is
$\Gamma=\tfrac12(I+J)$, where $J$ is the $4\times4$ all-ones
matrix. With $P_V:=B_VR_{\sigma_V}$, the terminal reference
matrix satisfies
\begin{equation}\label{eq:regular-terminal-gram}
 G_{\mathrm r}:=\frac{2^{56}}{245}P_V^{\tr}\Gamma P_V
 =\frac1{245}
 \begin{pmatrix}
 245&155&123&103\\155&197&85&105\\
 123&85&85&41\\103&105&41&77
 \end{pmatrix}.
\end{equation}
In particular, $G_{\mathrm r}>0$ and
$\det G_{\mathrm r}=1048576/720600125$.

Consider a comparison in the $k$-th phase of $W_{\hh,9}$. Let $a,b$ be
the coefficient vectors of the selected and competing edges in
the edge basis at the start of that block. At the normalized
reference return $(Q_{\hh,9}^n)^{\tr}G_{\mathrm r}Q_{\hh,9}^n$, the gap is
\[
 \alpha+\gamma z+\eta z^2,
 \qquad z=(-1/2)^n,
\]
where
\[
\begin{aligned}
 \alpha&=(E_{\hh,9}a)^{\tr}G_{\mathrm r}E_{\hh,9}a
          -(E_{\hh,9}b)^{\tr}G_{\mathrm r}E_{\hh,9}b,\\
 \gamma&=2\bigl((E_{\hh,9}a)^{\tr}G_{\mathrm r}F_{\hh,9}a
          -(E_{\hh,9}b)^{\tr}G_{\mathrm r}F_{\hh,9}b\bigr),\\
 \eta&=(F_{\hh,9}a)^{\tr}G_{\mathrm r}F_{\hh,9}a
          -(F_{\hh,9}b)^{\tr}G_{\mathrm r}F_{\hh,9}b.
\end{aligned}
\]
For the $81$ comparisons, direct substitution gives
\begin{equation}\label{eq:regular-tail-certificate}
\begin{aligned}
 \min\alpha&=\frac2{245},\\
 \min(\alpha+\gamma+\eta)&=\frac{31}{3920},\\
 \min\left(\alpha-\frac{|\gamma|}{2}
                    -\frac{|\eta|}{4}\right)&=\frac{59}{8820}.
\end{aligned}
\end{equation}
Since $z=1$ for $n=0$ and $|z|\leq1/2$ for $n\geq1$, every
reference gap is at least $\delta:=59/8820$.

It remains to control the perturbation through the prefix.
The Gram perturbation corresponding to
$\varepsilon(\beta+h)$ has each entry bounded in absolute value
by $45\varepsilon/4$, hence Euclidean operator norm at most
$45\varepsilon$. Since $\Gamma\succeq I/2$, the perturbed
initial Gram matrix $G_{\mathrm{in}}$ satisfies
\[
 -90\varepsilon\Gamma
 \preceq G_{\mathrm{in}}-\Gamma
 \preceq90\varepsilon\Gamma.
\]
In particular, it is positive definite for the stated range of
$\varepsilon$. Congruence by $P_V$, followed by the normalization
in~\eqref{eq:regular-terminal-gram}, preserves these inequalities.
If $G'$ denotes the resulting perturbed terminal Gram matrix,
then
\[
 -90\varepsilon G_{\mathrm r}
 \preceq G'-G_{\mathrm r}
 \preceq90\varepsilon G_{\mathrm r}.
\]
All diagonal entries of $G_{\mathrm r}$ are at most one. Factoring
through $G_{\mathrm r}^{1/2}$ therefore gives
\[
 \|G'-G_{\mathrm r}\|_{\max}\leq90\varepsilon.
\]
The uniform comparison estimate~\eqref{eq:four-uniform-comparison-error}
now bounds the change of every gap along the infinite tail by
$392\cdot90\varepsilon=35280\varepsilon$. Thus every gap is
at least
\[
 \frac{59}{8820}-35280\varepsilon>0
 \qquad(0<\varepsilon\leq10^{-7}).
\]
This proves infinite strict admissibility. Equation
\eqref{eq:open-four-normalized-powers} shows that the normalized
return matrices converge to $E_{\hh,9}^{\tr}G'E_{\hh,9}$, of rank three.
The degeneration argument in
Theorem~\ref{thm:open-four-dimensional}, including its treatment
of intermediate phases, applies unchanged.

Finally, for each fixed $\varepsilon>0$, the allowed $h$ range
over a full-dimensional open cube. Their union is open, proving
the last assertion.
\end{proof}

\begin{proof}[Proof of Corollary~\ref{cor:higher-dimensional-fsc}]
Let $X_0,\ldots,X_d$ be independent standard Gaussian vectors in
$\R^d$, and let $\Sigma_d=[X_0,\ldots,X_d]$. These vertices are
affinely independent almost surely. For five specified vertices,
put
\[
 D^{(d)}_{ij}:=\frac{|X_i-X_j|^2}{2d},
 \qquad 0\leq i<j\leq4.
\]
Then
\[
 \sqrt d\,(D^{(d)}-\mathbf1)
 =\frac1{\sqrt d}\sum_{\alpha=1}^dY_\alpha,
 \qquad
 (Y_\alpha)_{ij}
 =\frac12(X_{i,\alpha}-X_{j,\alpha})^2-1.
\]
The ten-dimensional vectors $Y_\alpha$ are independent,
identically distributed, and centered. Expanding their second
moments gives the covariance matrix $\Xi$ with entries
\[
 \Xi_{ef}=
 \begin{cases}
 2,&e=f,\\
 1/2,&e\ne f\text{ share a vertex},\\
 0,&e,f\text{ are disjoint}.
 \end{cases}
\]
If $C$ is the unsigned vertex-edge incidence matrix of $K_5$,
then
\[
 \Xi=I_{10}+\frac12C^{\tr}C>0.
\]
The multivariate central limit theorem
\cite[Theorem~3.10.7]{Durrett2019} therefore gives
\[
 \sqrt d\,(D^{(d)}-\mathbf1)
 \ \Longrightarrow\ Z,\qquad Z\sim N(0,\Xi).
\]
The cube
\[
 \Omega:=\left\{z\in\R^{10} \;\Big|\; \|z-\beta\|_\infty<\frac12 \right\}
\]
has positive $Z$-probability and its boundary has probability
zero. Hence, for some $p>0$ and all sufficiently large $d$,
\[
 \Pr\bigl(\sqrt d\,(D^{(d)}-\mathbf1)\in\Omega\bigr)
 \geq p.
\]
For such $d$ with $d^{-1/2}\leq10^{-7}$, this event implies
$D^{(d)}\in\mathcal K$. By
Lemma~\ref{lem:regular-four-cone}, the corresponding $4$-face
belongs to an open family having degenerating branches.

Partition the $d+1$ vertices into
$m=\lfloor(d+1)/5\rfloor$ disjoint groups of five, ignoring any
remaining vertices. The corresponding events are independent.
The probability that none occurs is at most $(1-p)^m$. If one
occurs, Lemma~\ref{lem:lift-face-branch} gives a degenerating
branch of $\Sigma_d$. Since the successful face condition is
open, $\Sigma_d$ lies in the interior of the set of simplices
having such a branch. We obtain, for an absolute $c_0>0$ and all
sufficiently large $d$,
\[
 \Pr\bigl(\Sigma_d\text{ lies in that interior}\bigr)
 \geq1-(1-p)^{\lfloor(d+1)/5\rfloor}
 \geq1-e^{-c_0d}.
\]
For each fixed $d\geq4$, Corollary~\ref{cor:open-all-dimensions}
provides a nonempty open set of successful vertex configurations.
The joint Gaussian density is positive everywhere, so this set
has positive probability. Reducing $c_0$ over the finitely many
remaining dimensions gives a single $c>0$ for which the stated
bound holds for every $d\geq4$.

Finally, each exceptional hyperplane from
Lemma~\ref{lem:full-tree-tie-free} becomes a nonzero polynomial
equation in the initial vertex coordinates. It is nonzero because
every positive-definite Gram matrix is realized by some simplex.
Such a polynomial zero set has Lebesgue measure zero. Absolute
continuity of the Gaussian law and countability of the equations
therefore show that the entire refinement tree has pairwise
distinct edge lengths almost surely.
\end{proof}

\section{The elliptic construction}\label{sec:elliptic}

In this section we prove Theorem~\ref{thm:main-elliptic} by
constructing a compact one-dimensional family of tetrahedra,
parametrized by an ellipse, which is preserved projectively by a
particular marked word. The
induced action on the parameter ellipse is conjugate to an irrational
rotation. We show that the whole family is nondegenerate, that the
word is LEB-admissible at every point of the family, and that every
resulting branch contains infinitely many similarity classes. 

Unlike in the hyperbolic construction of Section~\ref{sec:hyperbolic},
in the first construction, the opposite-edge tie-breaking rule does not apply in this construction.  Whenever several edges
are tied for longest, any one of them may be selected. The refined construction, however, does not need tie-breaking. 

\subsection{Method}
\label{subsec:elliptic-method}
The elliptic mechanism below is formulated in the squared-edge coordinates model because both the invariant affine family and the LEB-admissibility conditions have
particularly simple expressions in these coordinates. We pass to the
Gram-matrix model only when verifying nondegeneracy.

Let $\omega$ be a marked word. Suppose that there is a vector
$\mathbf{d}_*\in\mathcal C$ (see Definition~\ref{def:tet_cone_squared-edge}), a $6\times2$ matrix $V=
    \begin{bmatrix}
        v_1&v_2
    \end{bmatrix}$,
a scalar $\lambda>0$, and a matrix $U\in\GL(2,\R)$ such that
\begin{equation}\label{eq:elliptic_mechanism}
    M_\omega \mathbf{d}_*=\lambda \mathbf{d}_*,
    \qquad
    M_\omega V=\lambda VU.
\end{equation}
Then the affine plane
\[
    D(u):=\mathbf{d}_*+Vu,
    \qquad
    u\in\R^2,
\]
satisfies
\begin{equation}\label{eq:general_elliptic_return}
    M_\omega D(u)=\lambda D(Uu).
\end{equation}
After projectivization, the common factor $\lambda$ disappears, so
the marked word acts on the parameter plane $D$ by $U$. 
The elliptic mechanism occurs when $U$ preserves a positive-definite
quadratic form. That is, if $U^{\tr}SU=S$ for some $S>0$, then every level set
\[
    \left\{u\in\R^2\mid u^{\tr}Su=c\right\},
    \qquad c>0,
\]
is a compact $U$-invariant ellipse. If, in addition, the eigenvalues
of $U$ lie on the unit circle and are not roots of unity, then the
action on each such ellipse is conjugate to an irrational rotation,
and all its orbits are infinite.

To obtain an actual LEB branch, it remains to verify that one of these
invariant ellipses lies in the tetrahedral cone and in the
admissibility domain of the word. Compactness then keeps the marked
orbit, as well as the finitely many intermediate stages of each
block, uniformly inside the nondegenerate shape space.

\subsection{An elliptic family with tie-breaking}
\label{subsec:elliptic-main}
To find a marked word $\omega$ that satisfies the conditions described at the beginning of the previous subsection, we ran a brute-force search that produced the following.

Let
\begin{equation*}\label{eq:omega_e}
    W_{\ee,3}
    :=
    (0\leftarrow1)(1\leftarrow2)(3\leftarrow1),
    \qquad
    \sigma_{\ee,3}
    :=
    (1\,2\,3),
    \qquad
    \omega_{\ee,3}
    :=
    (W_{\ee,3},\sigma_{\ee,3}),   
\end{equation*}

where $\sigma_{\ee,3}$ fixes $0$.

\begin{lemma}\label{lem:elliptic-construction}
    For the marked word $\omega_{\ee,3}$ the following holds:
    \begin{enumerate}[(a)]
        \item 
        There exist
    \[
    \mathbf{d}_*\in\mathcal{C}, \quad \text{ a } 6\times 2 \text{ matrix }    V=\begin{bmatrix}
        v_1&v_2
    \end{bmatrix}, \quad 
    \lambda>0, \quad 
    \text{ and } \quad 
    U\in\GL(2,\R) 
    \]
    satisfying \eqref{eq:elliptic_mechanism} and \eqref{eq:general_elliptic_return}, with $\omega=\omega_{\ee,3}$. 
        \item 
        The matrix $U$ preserves a quadratic form that defines an ellipse $\mathscr E_{\ee,3}$, on which $U$ acts as an irrational rotation.
        \item
         For every $(x,y)\in\mathscr E_{\ee,3}$ we have $D(x,y):=\mathbf{d}_*+xv_1+yv_2\in\mathcal C$. 
    \end{enumerate}
    Consequently, $D(\mathscr E_{\ee,3})$ is a projectively
$\omega_{\ee,3}$-invariant compact family of nondegenerate
tetrahedra, and every $U$-orbit in $\mathscr E_{\ee,3}$ is
infinite.
\end{lemma}

\begin{proof}
We derive the objects in the order in which they arise from the
marked-word matrix. By Definition~\ref{def:marked-word-matrix}, the
word $\omega_{\ee,3}$ determines
\[
    M_{\ee,3}
    :=
    M_{\omega_{\ee,3}}
    =
    \Pi_{\sigma_{\ee,3}}
    L_{3\leftarrow1}L_{1\leftarrow2}L_{0\leftarrow1}.
\]
Using \eqref{eq:explicit_Pi} and \eqref{eq:explicit_L}, a direct calculation gives
\[
    M_{\ee,3}
    =
    \begin{pmatrix}
        -\frac14&\frac12&0&\frac12&0&0\\
        -\frac18&\frac18&\frac14&\frac1{16}&\frac18&-\frac18\\
        0&\frac14&0&0&0&0\\
        0&0&0&\frac3{16}&-\frac18&\frac38\\
        0&0&0&\frac14&0&0\\
        0&0&0&-\frac1{16}&\frac18&\frac18
    \end{pmatrix},
\]
whose characteristic polynomial is
\[
    \chi_{M_{\ee,3}}(z)
    =
    \frac1{4096}
    (4z-1)^2
    (16z^2+6z+1)
    (16z^2-z+1).
\]
We first choose the center $\mathbf{d}_*$ of the affine family $D(x,y)$. The eigenspace of
$M_{\ee,3}$ corresponding to $\lambda:=\frac14$ is
\[
    \ker\left(M_{\ee,3}-\frac14I\right)
    =
    \operatorname{span}
    \left\{
        \begin{pmatrix}1\\1\\1\\0\\0\\0\end{pmatrix},
        \begin{pmatrix}2\\0\\0\\2\\2\\1\end{pmatrix}
    \right\}.
\]
Inside this eigenspace, we take the vector $\mathbf{d}_*$ below, which is a convenient integral point in $\ker(M_{\ee,3}-\frac14I)\cap\mathcal C$, around which a small invariant ellipse fits:
\[
    \mathbf{d}_*
    :=
    3
    \begin{pmatrix}1\\1\\1\\0\\0\\0\end{pmatrix}
    +
    10
    \begin{pmatrix}2\\0\\0\\2\\2\\1\end{pmatrix}
    =
    \begin{pmatrix}23\\3\\3\\20\\20\\10\end{pmatrix}, \quad 
    \text{ thus }
    \quad
    M_{\ee,3} \mathbf{d}_*=\frac14\mathbf{d}_*.
\]
Next, construct the directions $v_1,v_2$ of the affine family $D(x,y)$. The roots of
the factor $16z^2+6z+1$ are
\[
    \mu,\overline\mu
    =
    \frac{-3\pm i\sqrt7}{16},
\]
and both have modulus $\frac14$. The corresponding real invariant plane
is
\[
    F
    :=
    \ker\left(16M_{\ee,3}^2+6M_{\ee,3}+I\right).
\]
A direct computation gives $F=\operatorname{span}\{v_1,v_2\}$, where
\[
    v_1
    :=
    \begin{pmatrix}4\\1\\0\\0\\0\\0\end{pmatrix},
    \qquad
    v_2
    :=
    \begin{pmatrix}4\\0\\1\\0\\0\\0\end{pmatrix}.
\]
Thus, setting $V:=\begin{bmatrix}v_1&v_2\end{bmatrix}$, the affine plane $D(x,y)$ can be written as
\begin{equation}\label{eq:D_e}
    D(x,y)
    := \mathbf{d}_*+
    V\begin{pmatrix}x\\y\end{pmatrix} =
    \begin{pmatrix}
        23+4x+4y\\
        3+x\\
        3+y\\
        20\\
        20\\
        10
    \end{pmatrix}.    
\end{equation}
The matrix $U$ is now obtained from the action of $M_{\ee,3}$ on
$F$. Indeed,
\[
    M_{\ee,3} v_1
    =
    -\frac38v_1+\frac14v_2,
    \qquad
    M_{\ee,3} v_2
    =
    -\frac14v_1.
\]
Therefore, the matrix of the normalized restriction
$4M_{\ee,3} |_F$ in the basis $(v_1,v_2)$ is
\[
    U
    =
    \begin{pmatrix}
        -\frac32&-1\\
        1&0
    \end{pmatrix}.
\]
Equivalently, $M_{\ee,3} V=\frac14VU$, which together with the previous $M_{\ee,3} \mathbf{d}_*=\frac14\mathbf{d}_*$ shows \eqref{eq:elliptic_mechanism}.
In particular, we obtain that 
\begin{equation}\label{eq:elliptic-return}
    M_{\ee,3} D(x,y)
    =
    M_{\ee,3} \mathbf{d}_*
    +
    M_{\ee,3} V
    \begin{pmatrix}x\\y\end{pmatrix}   =
    \frac14\mathbf{d}_*
    +
    \frac14VU
    \begin{pmatrix}x\\y\end{pmatrix}   =
    \frac14
    D\left(        U\begin{pmatrix}x\\y\end{pmatrix}
    \right),
\end{equation}
which proves item $(a)$.

Next, we construct an invariant ellipse from the quadratic form preserved by $U$. Write
$S=
    \begin{pmatrix}
        a&b\\
        b&c
    \end{pmatrix}$ and solve
$U^{\tr}SU=S$. This gives
$b=\frac34a$ and $c=a$. Choosing $a=4$, we obtain the positive-definite matrix $S
    :=
    \begin{pmatrix}
        4&3\\
        3&4
    \end{pmatrix}$.
Hence, the quadratic form
\[
    Q(x,y)
    :=
    \begin{pmatrix}x&y\end{pmatrix}
    S
    \begin{pmatrix}x\\y\end{pmatrix}
    =
    4x^2+6xy+4y^2,
\]
and thus every positive level set of $Q$, is $U$-invariant. The level is chosen only to ensure that the ellipse remains inside the tetrahedral cone and the admissibility region of the word.
Choosing the level $Q=\frac1{16}$ gives precisely the ellipse
\begin{equation}\label{eq:E_e}
    \mathscr E_{\ee,3}
    :=
    \left\{
        (x,y)\in\R^2 \;\middle|\;
        4x^2+6xy+4y^2=\frac1{16}
    \right\}.
\end{equation}

To identify the action on this ellipse, write $C^{\tr}C=S$, so
$C:=
    \begin{pmatrix}
        2&\frac32\\
        0&\frac{\sqrt7}{2}
    \end{pmatrix}$. 
     The change of coordinates $w=Cu$ maps
$\mathscr E_{\ee,3}$ to the Euclidean circle of radius $\frac14$, where $CUC^{-1}
    =
    \begin{pmatrix}
        -\frac34&-\frac{\sqrt7}{4}\\
        \frac{\sqrt7}{4}&-\frac34
    \end{pmatrix}$ acts via rotation by an angle $\theta$ satisfying $e^{i\theta}
    =
    \frac{-3+i\sqrt7}{4}$.
    
We claim that this number is not a root of unity. Indeed, since its modulus is one,
its inverse is its complex conjugate, and hence $e^{i\theta}+e^{-i\theta}=-\frac32$. If $e^{i\theta}$ were a root of unity, then it and its inverse would
be algebraic integers, and so would their sum. This is impossible
because every rational algebraic integer is an integer. Therefore $\frac{\theta}{2\pi}\notin\Q$, and the action of $U$ on $\mathscr E_{\ee,3}$ is conjugate to an
irrational rotation. In particular, every orbit is dense in
$\mathscr E_{\ee,3}$, and hence infinite, proving item $(b)$.

It remains to verify that the points of the ellipse represent
nondegenerate tetrahedra. For that we will use the Gram-matrix model. The eigenvalues of $S$ are $1$ and $7$.
Therefore, by \eqref{eq:E_e}, for $(x,y)\in\mathscr E_{\ee,3}$ we have
\[
    x^2+y^2
    \leq
    Q(x,y)
    =
    \frac1{16},
\]
and hence
\begin{equation}\label{eq:elliptic-coordinate-bound}
    |x|,|y|\leq\frac14.
\end{equation}
By Proposition~\ref{prop:gram-edge-equivalence} and by applying \eqref{eq:E_inverse}, the Gram matrix
corresponding to $D(x,y)$ is
\begin{equation}\label{eq:G_e}
    G_{\ee,3}(x,y)
    =
    \begin{pmatrix}
        23+4x+4y
        &
        3+\frac52x+2y
        &
        3+2x+\frac52y
        \\[1mm]
        3+\frac52x+2y
        &
        3+x
        &
        -2+\frac{x+y}{2}
        \\[1mm]
        3+2x+\frac52y
        &
        -2+\frac{x+y}{2}
        &
        3+y
    \end{pmatrix}.
\end{equation}
Its leading principal minors are
\[
    m_1=23+4x+4y,
\]
\[
    m_2
    =
    \frac{240+80x-9x^2-24xy-16y^2}{4},
\]
and
\[
    m_3
    =
    25-\frac{35}{4}Q(x,y).
\]
Using \eqref{eq:elliptic-coordinate-bound}, we obtain
\[
    m_1\geq21,
    \qquad
    4m_2
    =
    240+80x-(3x+4y)^2
    \geq
    220-\frac{49}{16}>0, \qquad
    m_3
    =
    25-\frac{35}{64}>0.
\]
Sylvester's criterion therefore gives
$G_{\ee,3}(x,y)>0$. By Definition~\ref{def:tet_cone_squared-edge}, this is equivalent to
$D(x,y)\in\mathcal C$, which proves item $(c)$.

Finally, $\mathscr E_{\ee,3}$ is compact and $D$ is continuous,
so $D(\mathscr E_{\ee,3})$ is a compact subset of $\mathcal C$.
Its image in
$\PP\mathcal C\cong\mathcal T_{\mathrm{lab}}$ is therefore compact. The return formula
\eqref{eq:elliptic-return}, together with
$U(\mathscr E_{\ee,3})=\mathscr E_{\ee,3}$, shows that this
family is projectively $\omega_{\ee,3}$-invariant.
\end{proof}

For $(x,y)\in\mathscr E_{\ee,3}$, let
$\Delta_{x,y}$ denote a labeled tetrahedron whose squared-edge vector
is $D(x,y)$.

\begin{lemma}\label{lem:elliptic-admissibility}
For every $(x,y)\in\mathscr E_{\ee,3}$, the bisections induced by
\[
    W_{\ee,3}
    =
    (0\leftarrow1)(1\leftarrow2)(3\leftarrow1)
\]
on $\Delta_{x,y}$ are LEB-admissible.
\end{lemma}

\begin{proof}
A direct computation gives that 
\[
    L_{0\leftarrow1}
    =
    \begin{pmatrix}
        \frac14&0&0&0&0&0\\
        -\frac14&\frac12&0&\frac12&0&0\\
        -\frac14&0&\frac12&0&\frac12&0\\
        0&0&0&1&0&0\\
        0&0&0&0&1&0\\
        0&0&0&0&0&1
    \end{pmatrix}, 
    \text{ and }
    L_{1\leftarrow2}
    =
    \begin{pmatrix}
        \frac12&\frac12&0&-\frac14&0&0\\
        0&1&0&0&0&0\\
        0&0&1&0&0&0\\
        0&0&0&\frac14&0&0\\
        0&0&0&-\frac14&\frac12&\frac12\\
        0&0&0&0&0&1
    \end{pmatrix}.
\]
By \eqref{eq:D_e}, denoting $D^{(1)}(x,y)
    =
    L_{0\leftarrow1}D(x,y)$ and $D^{(2)}(x,y)
    =
    L_{1\leftarrow2}D^{(1)}(x,y)$ yields 
\[
    D^{(1)}(x,y)
    =
    \begin{pmatrix}
        x+y+\frac{23}{4}\\
        -\frac x2-y+\frac{23}{4}\\
        -x-\frac y2+\frac{23}{4}\\
        20\\
        20\\
        10
    \end{pmatrix},
\text{ and }
    D^{(2)}(x,y)
    =
    \begin{pmatrix}
        \frac x4+\frac34\\
        -\frac x2-y+\frac{23}{4}\\
        -x-\frac y2+\frac{23}{4}\\
        5\\
        10\\
        10
    \end{pmatrix}.
\]
By \eqref{eq:elliptic-coordinate-bound}, the edge $01$ is the unique
longest edge in $D(x,y)$; the edges $12$ and $13$ are tied for
longest in $D^{(1)}(x,y)$; and the edges $13$ and $23$ are tied for
longest in $D^{(2)}(x,y)$. Thus the selected edges
\[
    01,
    \qquad
    12,
    \qquad
    13
\]
are longest at the three successive stages, proving admissibility.
\end{proof}

\begin{remark}\label{rem:elliptic-ties}
At the second bisection of each elliptic block, the edges $12$ and
$13$ are tied for longest. Their opposite edges are $03$ and $02$,
respectively, and
\[
    d_{03}^{(1)}-d_{02}^{(1)}=\frac{y-x}{2}.
\]
The prescribed word $W_{\ee,3}$ always chooses the edge $12$ at this stage, which agrees with the opposite-edge rule only if $y\ge x$. Since the dynamics is conjugate to an irrational rotation,
every orbit visits both regions $y>x$ and $x>y$ infinitely often. Consequently, $W_{\ee,3}$ does not define a branch governed by the
opposite-edge tie-breaking rule. 

We note that at the third bisection, $13$ and $23$ are tied for longest,
and
\[
    d_{02}^{(2)}-d_{01}^{(2)}
    =
    5-\frac34x-y>0.
\]
Thus the opposite-edge rule always agrees with the prescribed edge $13$
at this stage.
\end{remark}

\begin{lemma}\label{lem:elliptic-non-similar}
For $(x,y),(x',y')\in\mathscr E_{\ee,3}$, the tetrahedra
$\Delta_{x,y}$ and $\Delta_{x',y'}$ are similar as unlabeled
tetrahedra if and only if
\[
    (x',y')=(x,y)
    \qquad\text{or}\qquad
    (x',y')=(y,x).
\]
Consequently, the family
\[
    \left\{
        \Delta_{x,y}\mid
        (x,y)\in\mathscr E_{\ee,3},\ x>y
    \right\}
\]
is uncountable and consists of pairwise non-similar tetrahedra.
Moreover, every $U$-orbit on $\mathscr E_{\ee,3}$ gives
infinitely many unlabeled similarity classes.
\end{lemma}

\begin{proof}
By \eqref{eq:elliptic-coordinate-bound}, the six squared-edge lengths
of $\Delta_{x,y}$ have the following order:
\[
    23+4x+4y>20=20>10>
    \max\{3+x,3+y\}.
\]
Suppose that $\Delta_{x,y}$ and $\Delta_{x',y'}$ are similar. Since a
similarity rescales all squared-edge lengths by the same positive
factor, and in both tetrahedra the second- and third-largest squared
edge lengths are equal to $20$, that factor must be $1$. Comparing
the two smallest squared-edge lengths then gives
\[
    \{x,y\}=\{x',y'\}.
\]
Thus $(x',y')=(x,y)$ or $(x',y')=(y,x)$. Conversely, the transposition
of the vertices $2$ and $3$ takes $D(x,y)$ to $D(y,x)$.

Finally, every
$U$-orbit is infinite by Lemma~\ref{lem:elliptic-construction}, while
an unlabeled similarity class contains at most the two parameter
points $(x,y)$ and $(y,x)$. Therefore every such orbit contains
infinitely many unlabeled similarity classes.
\end{proof}

\begin{theorem}\label{thm:elliptic-family}
For every $(x_0,y_0)\in\mathscr E_{\ee,3}$, repeated application
of the marked word $\omega_{\ee,3}$ produces
an infinite admissible LEB branch. The branch remains in a compact
subset of the nondegenerate tetrahedral shape space and contains
infinitely many pairwise non-similar tetrahedra.
\end{theorem}

\begin{proof}
Set
\[
    \binom{x_n}{y_n}
    :=
    U^n\binom{x_0}{y_0}.
\]
By Lemma~\ref{lem:elliptic-construction},
$(x_n,y_n)\in\mathscr E_{\ee,3}$ for every $n$, and iterating
\eqref{eq:elliptic-return} gives
\[
    M_{\omega_{\ee,3}}^nD(x_0,y_0)
    =
    4^{-n}D(x_n,y_n).
\]
Thus, after $n$ marked blocks, the resulting tetrahedron is similar
to $\Delta_{x_n,y_n}$. Lemma~\ref{lem:elliptic-admissibility} applies
at every parameter point $(x_n,y_n)$, so the block can be iterated
indefinitely and produces an admissible LEB branch. By
Lemma~\ref{lem:elliptic-non-similar}, its marked subsequence contains
infinitely many pairwise non-similar tetrahedra.

The complete branch, including the two intermediate tetrahedra in
each block, lies, after resetting labels at the block boundaries, in
the projectivization of the finite union
\[
    D(\mathscr E_{\ee,3})
    \cup
    L_{0\leftarrow1}D(\mathscr E_{\ee,3})
    \cup
    L_{1\leftarrow2}L_{0\leftarrow1}
    D(\mathscr E_{\ee,3}).
\]
Since $D(\mathscr E_{\ee,3})$ is compact and the maps
$L_{0\leftarrow1}$ and $L_{1\leftarrow2}$ are continuous, all three
sets in the displayed union are compact. Moreover,
$L_b(\mathcal C)=\mathcal C$ for every directed bisection $b$, so all
three sets are contained in $\mathcal C$. Their projectivizations
therefore form a compact subset of
$\mathcal T_{\mathrm{lab}}$, and hence also of the unlabeled shape
space $\mathcal T$.
\end{proof}

\subsection{An elliptic family without tie-breaking}

To finish our proof of Theorem~\ref{thm:main-elliptic}, we now perturb to a non-tie-breaking family. Unlike the
length-three word in Subsection~\ref{subsec:elliptic-main}, the word
below has strict longest-edge comparisons on an open invariant
neighborhood. An explicit algebraic perturbation will also remove
ties from the entire refinement tree.

Let $\omega_{\ee,9}=(W_{\ee,9},\sigma_{\ee,9})$, where
\begin{equation}\label{eq:strict-elliptic-word}
\begin{aligned}
 W_{\ee,9}:={}&
 (3\leftarrow0)(0\leftarrow2)(1\leftarrow2)
 (2\leftarrow3)(3\leftarrow1)\\
 & (1\leftarrow0)(3\leftarrow0)(0\leftarrow2)
 (2\leftarrow3),
 \qquad \sigma_{\ee,9}:=(0\,1\,2\,3).
\end{aligned}
\end{equation}
Thus the final relabeling is
$(p_0,p_1,p_2,p_3)\mapsto(p_1,p_2,p_3,p_0)$. Put
\begin{equation}\label{eq:strict-elliptic-reference}
 H_{\ee,9}:=
 \begin{pmatrix}38&31&38\\31&44&40\\38&40&56\end{pmatrix},
 \qquad
 G_{\ee,9}:=H_{\ee,9}+e_1e_1^{\tr}
 =\begin{pmatrix}39&31&38\\31&44&40\\38&40&56\end{pmatrix}.
\end{equation}
The leading principal minors of $H_{\ee,9}$ are $38,711,9720$,
so both matrices are positive definite. The squared-edge vector
of $G_{\ee,9}$ is
\[
 \mathbf{d}_{\ee,9}=(39,44,56,21,19,20)^{\tr}.
\]
The marked return and its normalization are
\begin{equation}\label{eq:strict-elliptic-return-matrix}
 P_{\ee,9}:=P_{\omega_{\ee,9}}
 =\frac1{16}
 \begin{pmatrix}-3&-2&-4\\-1&-2&0\\3&2&2\end{pmatrix},
 \qquad Q_{\ee,9}:=8P_{\ee,9}.
\end{equation}
Direct multiplication gives
\begin{equation}\label{eq:strict-elliptic-isometry}
 Q_{\ee,9}^{\tr}H_{\ee,9}Q_{\ee,9}=H_{\ee,9},
 \qquad \det Q_{\ee,9}=-1,
\end{equation}
and
\[
 \chi_{Q_{\ee,9}}(z)=(z+1)(z^2+\tfrac{z}2+1).
\]
Write
\[
 \zeta:=\frac{-1+i\sqrt{15}}4.
\]
The eigenvalues of $Q_{\ee,9}$ are $-1,\zeta,\overline\zeta$.
The number $\zeta$ is not a root of unity: otherwise
$\zeta+\zeta^{-1}=-\frac12$ would be a rational algebraic integer.
Thus the normalized return has an irrational rotational component.

\begin{lemma}\label{lem:strict-elliptic-open}
Let
\begin{equation}\label{eq:strict-elliptic-neighborhood}
 \mathcal U_{\ee,9}:=
 \left\{G\in\Sym_3 \Bigg |
 \frac{99}{100}H_{\ee,9}\prec G
 \prec\frac{989}{900}H_{\ee,9}\right\}.
\end{equation}
Every $G\in\mathcal U_{\ee,9}$ admits infinite repetition of
$\omega_{\ee,9}$, with a uniquely longest edge at every
bisection. The entire branch remains in a compact subset of the
nondegenerate shape space. Its marked orbit contains infinitely
many unlabeled similarity classes unless
$\mathscr C_{Q_{\ee,9}}(G)=G$.
\end{lemma}

\begin{proof}
For brevity, write $H=H_{\ee,9}$ and $Q=Q_{\ee,9}$ in this
proof. Applying the word to the reference Gram matrix $H$ gives
the following selected squared lengths $a_k$ and largest
competing squared lengths $b_k$:
\[
\begin{array}{c|c|c|c}
 k&\text{bisection}&a_k&b_k\\\hline
 1&3\leftarrow0&56&44\\
 2&0\leftarrow2&44&38\\
 3&1\leftarrow2&20&18\\
 4&2\leftarrow3&18&11\\
 5&3\leftarrow1&11&19/2\\
 6&1\leftarrow0&19/2&9/2\\
 7&3\leftarrow0&9/2&7/2\\
 8&0\leftarrow2&7/2&19/8\\
 9&2\leftarrow3&5/4&9/8
\end{array}
\]
In particular, $a_k\geq \frac{10}{9}b_k$ for every $k$.

By~\eqref{eq:strict-elliptic-isometry}, $\mathcal U_{\ee,9}$ is
invariant under $\mathscr C_Q$. Congruence by every prefix of the word
preserves its two defining inequalities. Hence, at phase $k$ of
every normalized block, the selected squared length is greater
than $\frac{99}{100}a_k$, while every competitor is less than
$(989/900)b_k$. The gap is therefore greater than
\[
 \frac{99}{100}a_k-\frac{989}{900}b_k
 \geq\frac1{900}b_k>0.
\]
This proves strict admissibility of the entire branch.

The closure of $\mathcal U_{\ee,9}$ is a compact subset of
$\Sym_3^{>0}$. Its images under the finitely many prefix
congruences are also compact and positive definite. Their
projectivizations contain every phase of the branch, proving
uniform nondegeneracy.

Since $Q$ is diagonalizable over $\C$, the induced congruence
operator $\mathscr C_Q$ is diagonalizable as well. Its eigenvalues are
\[
 1,1,-\zeta,-\overline\zeta,\zeta^2,\overline\zeta^{\,2}.
\]
None of the last four is a root of unity. Consequently,
$\mathscr C_Q^m(G)=G$ for some $m\geq1$ holds precisely when $\mathscr C_Q(G)=G$.
Furthermore, $\det \mathscr C_Q^n(G)=\det G$, so equality of projective
marked iterates implies equality of their normalized Gram
matrices. Thus a nonfixed $G$ has an infinite labeled similarity
orbit. Since vertex relabeling is a finite quotient, this orbit
contains infinitely many unlabeled similarity classes as well.
\end{proof}

We identify the elliptic orbit of the explicit matrix
$G_{\ee,9}$. Define
\[
 q(a,b):=(a,b,2b)^{\tr},
 \qquad
 U_{\ee,9}:=\begin{pmatrix}-3/2&5/2\\-1&1\end{pmatrix},
\]
\[
 \mathscr E_{\ee,9}:=
 \{(a,b)\in\R^2\mid 4a^2-10ab+10b^2=4\}.
\]
The identities
\[
 Q_{\ee,9}^{\tr}q(a,b)=q\bigl(U_{\ee,9}(a,b)^{\tr}\bigr),
 \qquad
 U_{\ee,9}^{\tr}
 \begin{pmatrix}4&-5\\-5&10\end{pmatrix}
 U_{\ee,9}=
 \begin{pmatrix}4&-5\\-5&10\end{pmatrix}
\]
show that
\begin{equation}\label{eq:strict-elliptic-curve}
 \mathcal G_{\ee,9}(a,b):=
 H_{\ee,9}+q(a,b)q(a,b)^{\tr},
 \qquad (a,b)\in\mathscr E_{\ee,9},
\end{equation}
is a projectively invariant family. More explicitly, if
$(a',b')^{\tr}=U_{\ee,9}(a,b)^{\tr}$, then
\[
 \mathscr C_{P_{\ee,9}}(\mathcal G_{\ee,9}(a,b))
 =\frac1{64}\mathcal G_{\ee,9}(a',b').
\]
A direct calculation gives
\[
 q(a,b)^{\tr}H_{\ee,9}^{-1}q(a,b)
 =\frac{4a^2-10ab+10b^2}{45}.
\]
Thus, on $\mathscr E_{\ee,9}$,
\begin{equation}\label{eq:strict-elliptic-rank-one-bound}
 H_{\ee,9}\preceq\mathcal G_{\ee,9}(a,b)
 \preceq\frac{49}{45}H_{\ee,9},
\end{equation}
so the whole family lies in $\mathcal U_{\ee,9}$.

The eigenvalues of $U_{\ee,9}$ are $\zeta,\overline\zeta$, and its
action on $\mathscr E_{\ee,9}$ is conjugate to an irrational rotation.
The map $\mathcal G_{\ee,9}$ identifies exactly antipodal parameter
points. Its image is an affine ellipse: after a linear change
sending $\mathscr E_{\ee,9}$ to a circle, its matrix entries are affine
functions of the doubled-angle coordinates. Thus the induced
return on $\mathcal G_{\ee,9}(\mathscr E_{\ee,9})$ is again conjugate to an
irrational rotation. Also,
\[
 \det\mathcal G_{\ee,9}(a,b)
 =9720\left(1+\frac4{45}\right)=10584,
\]
so projectivization introduces no additional identifications.
In particular, $(a,b)=(1,0)$ gives the explicit example
$G_{\ee,9}$, and every marked orbit on this ellipse is infinite.

\begin{proof}[Proof of Theorem~\ref{thm:main-elliptic}]
The family~\eqref{eq:strict-elliptic-curve} is uncountable, while
each unlabeled similarity class contains at most finitely many
of its parameter points. It therefore contains an uncountable
subfamily of pairwise non-similar tetrahedra. By
Lemma~\ref{lem:strict-elliptic-open}, all their branches are
strictly admissible, remain uniformly nondegenerate, and contain
infinitely many similarity classes.

More generally, the fixed space of $\mathscr C_{Q_{\ee,9}}$ has dimension
two. Removing it from $\mathcal U_{\ee,9}$ gives a nonempty open
set of Gram matrices having strictly admissible, uniformly
nondegenerate infinite marked orbits.
Finally, by  Lemma~\ref{lem:full-tree-tie-free}, the assertion follows. 
\end{proof}

We finish with an explicit example for which no ties occur anywhere
in the refinement tree, rather than only along the elliptic branch.
Let
\[
 \tau:=\frac{\sqrt[7]{2}}2,
\]
and prescribe the squared-edge vector
\begin{equation}\label{eq:elliptic-full-tree-example}
 \mathbf{d}_\tau:=
 \begin{pmatrix}39\\44\\56\\21\\19\\20\end{pmatrix}
 +\frac1{1000}
 \begin{pmatrix}\tau\\\tau^2\\\tau^3\\\tau^4\\\tau^5\\\tau^6\end{pmatrix}.
\end{equation}

\begin{proposition}\label{prop:elliptic-full-tree-example}
The vector $\mathbf{d}_\tau$ represents a nondegenerate tetrahedron with a
strictly admissible LEB branch that remains uniformly
nondegenerate and contains infinitely many similarity classes.
Every descendant in its entire refinement tree has pairwise
distinct edge lengths.
\end{proposition}

\begin{proof}
Let $G_\tau=\mathcal I^{-1}(\mathbf{d}_\tau)$ and
$D=G_\tau-G_{\ee,9}$. Since $0<\tau<1$, formula
\eqref{eq:E_inverse} gives
\[
 \|D\|_{\max}<\frac3{2000},
 \qquad \|D\|_{\mathrm{op}}<\frac9{2000}.
\]
The leading principal minors of $H_{\ee,9}-4I$ are
$34,399,2828$, so $H_{\ee,9}\succ4I$. Hence
\[
 -\frac9{8000}H_{\ee,9}\prec D
 \prec\frac9{8000}H_{\ee,9}.
\]
Together with~\eqref{eq:strict-elliptic-rank-one-bound} at
$(a,b)=(1,0)$, these inequalities imply
$G_\tau\in\mathcal U_{\ee,9}$.

The $(3,3)$ entry of $\mathscr C_{Q_{\ee,9}}(G)=G$ forces
$g_{11}=g_{13}$. In the present example,
\[
 (G_\tau)_{11}-(G_\tau)_{13}
 =1+\frac{\tau-\tau^3+\tau^5}{2000}>0.
\]
Thus $G_\tau$ is not fixed, and
Lemma~\ref{lem:strict-elliptic-open} proves the assertions about
its prescribed branch.

Finally, $\tau$ has degree seven over $\Q$, since $2\tau$ is a
root of $X^7-2$, which is Eisenstein at $2$. A rational linear
relation among the six coordinates of $\mathbf{d}_\tau$ would give a
polynomial of degree at most six vanishing at $\tau$. That
polynomial must be identically zero; its coefficients of
$\tau,\ldots,\tau^6$ then force the original relation to be zero.
The six squared-edge coordinates are therefore linearly
independent over $\Q$. Lemma~\ref{lem:full-tree-tie-free} excludes
all descendant edge equalities, even when arbitrary directed
bisections are allowed.
\end{proof}

\bibliographystyle{myamsalpha}
\bibliography{3d_LEB_degen}

\end{document}